\documentclass{amsart}

\usepackage{amsmath, amssymb, amsthm}
\usepackage{hyperref}
\usepackage[a4paper, bottom=2cm]{geometry}
\usepackage{tikz}
\usepackage{float}

\usepackage{enumitem}
\usepackage{xcolor}

\DeclareMathOperator{\Hom}{Hom}
\DeclareMathOperator{\Aut}{Aut}
\DeclareMathOperator{\Inn}{Inn}
\DeclareMathOperator{\Out}{Out}
\DeclareMathOperator{\Syl}{Syl}
\DeclareMathOperator{\Sym}{Sym}

\DeclareMathOperator{\Alt}{Alt}

\renewcommand{\L}{\operatorname{L}}
\renewcommand{\O}{\operatorname{O}}

\DeclareMathOperator{\U}{U}

\DeclareMathOperator{\GL}{GL}
\DeclareMathOperator{\SL}{SL}

\DeclareMathOperator{\SO}{SO}
\DeclareMathOperator{\Sp}{Sp}

\DeclareMathOperator{\Dih}{Dih}

\DeclareMathOperator{\GF}{GF}

\DeclareMathOperator{\Fi}{Fi}
\DeclareMathOperator{\Co}{Co}

\newcommand{\ON}{\mathrm{O}'\mathrm{N}}

\DeclareMathOperator{\He}{He}
\DeclareMathOperator{\Bm}{B}

\DeclareMathOperator{\Mn}{M}
\DeclareMathOperator{\Mt}{M}

\newcommand{\TE}{{^2 \mathrm{E}_6(2)}}

\newcommand{\bfA}{\mathbf{A}}

\newcommand{\bfU}{\mathbf{U}}
\newcommand{\bfT}{\mathbf{T}}
\newcommand{\bfJ}{\mathbf{J}}
\newcommand{\bfW}{\mathbf{W}}
\newcommand{\bfV}{\mathbf{V}}

\newcommand{\calB}{\mathcal{B}}

\newcommand{\calE}{\mathcal{E}}
\newcommand{\calF}{\mathcal{F}}
\newcommand{\calG}{\mathcal{G}}
\newcommand{\calH}{\mathcal{H}}

\newcommand{\calM}{\mathcal{M}}

\newcommand{\hyp}{\mathfrak{hyp}}
\newcommand{\foc}{\mathfrak{foc}}
\newcommand{\red}{\mathfrak{red}}

\newcommand{\normalIn}{\trianglelefteq}

\theoremstyle{definition}

\newtheorem{definition}{Definition}[section]

\newtheorem{notation}[definition]{Notation}

\theoremstyle{plain}
\newtheorem{theorem}[definition]{Theorem}
\newtheorem{lemma}[definition]{Lemma}
\newtheorem{proposition}[definition]{Proposition}

\title{Fusion Systems on Sylow $3$-subgroups of Fischer and Monster sporadic groups -- I}
\author{Pete Gautam}
\address{Department of Mathematics, University of Manchester, Manchester, M13 9PL, United Kingdom}
\email{pratyush.gautam@manchester.ac.uk}


\begin{document}
    \begin{abstract}
        We classify all corefree fusion systems on a Sylow $3$-subgroup of the sporadic groups $\Fi_{22}$, $\Fi_{23}$ and $\Bm$. We show that the $3$-group in each case does not support exotic fusion systems. This is the first of two papers that will complete the classification of all corefree fusion systems on Sylow $p$-subgroups of sporadic groups for $p$ odd.
    \end{abstract}
    
    \maketitle

    \section{Introduction}

    For a finite group $G$ and a $p$-subgroup $S$ of $G$, the \emph{fusion category} $\calF_S(G)$ is a category on the subgroups of $S$ with morphisms the conjugation maps coming from $G$. A \emph{fusion system} $\calF$ is a generalization of fusion categories defined solely in terms of the $p$-group $S$. We are particularly interested in \emph{saturated fusion systems}, which include all \emph{realizable} fusion systems $\calF_S(G)$, where $S$ is further assumed to be a Sylow $p$-subgroup of $G$. Much interest is in the case when a saturated fusion system is not realized by any finite group $G$, called an \emph{exotic} fusion system (for a discussion, see \cite[Problem III.7.4]{ako}).

    For a given $p$-group $S$, it is quite unlikely that $S$ supports exotic fusion systems (see, for example, \cite{algorithms}). On the other hand, many of the known exotic fusion systems have been found on Sylow $p$-subgroups of finite simple sporadic groups. In fact, the first examples of exotic fusion systems for odd primes were found in \cite{rv} on the extraspecial group $7^{1+2}_+$. This $7$-group is a Sylow $7$-subgroup of the sporadic groups $\ON$, $\He$ and $\Fi_{24}'$. Later work has resulted in further exotic fusion systems being found on Sylow $p$-subgroups of sporadic groups in \cite{abelian-1, abelian-2, ps18, vb-sporadics}. Moreover, the smallest exotic fusion system on a $2$-group is known to be supported on a Sylow $2$-subgroup of $\Co_3$ (see \cite{sol,ben,lo-spin7q}).

    Looking at \cite[Table 1]{vb-sporadics}, to complete the classification of fusion systems on Sylow $p$-subgroups of sporadic groups for $p$ odd, it remains to consider the Sylow $3$-subgroups of large sporadic groups. These are the Fischer groups $\Fi_{22}$, $\Fi_{23}$, $\Fi_{24}'$, the Baby Monster group $\Bm$ and the Monster group $\Mn$. We will finalize this classification over the course of two papers. 
    
    The first result we prove in this paper pertains to a Sylow $3$-subgroup of $\Fi_{22}$. We note that a Sylow $3$-subgroup of $\Fi_{22}$ is also isomorphic to a Sylow $3$-subgroup of $\O_7(3)$ and $\TE$.  The following result is stated as Theorem \ref{thm:fi22}.
    \begin{theorem}
        Let $\calF$ be a saturated fusion system on a Sylow $3$-subgroup of $\Fi_{22}$. If $O_3(\calF) = 1$, then $\calF$ is realized by an almost simple group $G$, where $O^{3'}(G) = F^*(G)$ is isomorphic to $\O_7(3)$, $\Fi_{22}$ or $\TE$.
    \end{theorem}

    A Sylow $3$-subgroup $S$ of $\Fi_{22}$ supports plenty of saturated fusion systems. As such, it is not practical to classify all saturated fusion systems on $S$. We shall only focus on fusion systems $\calF$ containing no normal $3$-subgroups, called \emph{corefree} fusion systems and denoted $O_3(\calF) = 1$.
    
    The outer automorphism groups of the simple groups $\O_7(3)$ and $\Fi_{22}$ have order $2$. On the other hand, $\TE$ has an outer automorphism group isomorphic to $\Sym(3)$. In particular, a Sylow $3$-subgroup of $\Aut(\TE)$ is not isomorphic to a Sylow $3$-subgroup of $\TE$. We extend the work on a Sylow $3$-subgroup of $\Fi_{22}$ to classify all corefree fusion systems on a Sylow $3$-subgroup $\Aut(\TE)$. This result is stated as Theorem \ref{thm:2e62}.
    
    \begin{theorem}
        Let $\calF$ be a saturated fusion system on a Sylow $3$-subgroup of $\Aut(\TE)$. If $O_3(\calF) = 1$, then $\calF$ is realized by $\TE : 3$ or $\Aut(\TE)$.
    \end{theorem}

    The next group we consider is $\Fi_{23}$. In this case, a Sylow $3$-subgroup of $\Fi_{23}$ is isomorphic to a Sylow $3$-subgroup of $\Bm$ and $\Aut(\O_8^+(3))$. The following result, given as Theorem \ref{thm:fi23}, classifies all the fusion systems of interest.
    \begin{theorem} \label{thm-c}
        Let $\calF$ be a saturated fusion system on a Sylow $3$-subgroup of $\Fi_{23}$. If $O_3(\calF) = 1$, then $\calF$ is realized by an almost simple group $G$, where either $O^{3'}(G) = F^*(G)$ and $G$ is isomorphic to $\Fi_{23}$ or $\Bm$, or $O^{3'}(G)$ is isomorphic to $\O^+_8(3) : 3$ or $\O^+_8(3) : \Alt(4)$.
    \end{theorem}

    The outer automorphism group of $\O_8^+(3)$ is isomorphic to $\Sym(4)$. To prove Theorem \ref{thm-c} for certain fusion systems $\calF$ realized by a subgroup of $\Aut(\O_8^+(3))$, we consider the normal subsystem $O^{3'}(O^3(\calF))$ of $\calF$ defined on a Sylow $3$-subgroup of $\O^+_8(3)$. As such, to understand the fusion system $\calF$ in these cases, we require the classification of all corefree fusion systems on a Sylow $3$-subgroup of $\O^+_8(3)$.
    
    \begin{theorem}
        Let $\calF$ be a saturated fusion system on a Sylow $3$-subgroup of $\O^+_8(3)$. If $O_3(\calF) = 1$, then $\calF$ is realized by an almost simple group $G$, where $O^{3'}(G) = F^*(G)$ is isomorphic to $\O^+_8(3)$.
    \end{theorem}

    The sporadic groups $\Fi_{22}$ and $\Fi_{23}$ are closely related to groups of Lie type in characteristic $3$. Indeed, $\Fi_{22}$ and $\O_7(3)$ have isomorphic Sylow $3$-subgroups, and the same holds for $\Fi_{23}$ and $\Aut(\O^+_8(3))$. In \cite[Conjecture 2]{ps-magma}, Parker and Semeraro conjecture that if $G$ is a group of Lie type in defining characteristic, not isomorphic to $\Sp_4(q)$ or a small number of exceptional cases, all corefree fusion systems on a Sylow $p$-subgroup of $G$ can be realized by a subgroup $G \leq H \leq \Aut(G)$. These two $3$-groups are certainly exceptional cases. For $p=2$, all exceptional cases are believed to have been classified following the recent work by Stroth in \cite{paper:stroth-o82}.

    In \cite{vb-sporadics}, van Beek states that a major obstacle in completing the classification of fusion systems related to the remaining sporadic groups would be computing with such large $3$-groups. The process of finding all fusion systems on a $p$-group $S$ starts with finding the \emph{proto-essential} subgroups of $S$. These are the subgroups of $S$ that could be \emph{essential} in some fusion system $\calF$ on $S$ (cf. Definition \ref{def:ess}). Finding all proto-essential subgroups is the first step in determining all saturated fusion systems on a given $p$-group as a consequence of the Alperin-Goldschmidt Theorem. 
    
    Prior to this work, the state-of-the-art algorithm to compute the proto-essential subgroups for a $p$-group was given in the Parker-Semeraro MAGMA package \cite{ps-magma}. This code was indeed quite valuable in their work \cite{algorithms} that classified fusion systems on groups of small order. The code computes all the proto-essential subgroups of $S$ by first finding all subgroups of the $p$-group $S/Z(S)$. 

    We were able to implement in GAP an equivalent code that computes the proto-essential subgroups of a $p$-group $S$. This is more efficient than the code in MAGMA since the subgroups are computed up to $\Aut(S)$-conjugacy, not $S$-conjugacy. This change is enough for us to find all the proto-essential subgroups in a Sylow $3$-subgroup of $\Fi_{22}$ and $\Fi_{23}$. However, this approach does not work for the Sylow $3$-subgroups of $\Fi_{24}$ and $\Mn$. Indeed, it is intractable to construct all the relevant subgroups for these $3$-groups due to their large orders.
    
    In a sequel paper, we will present an algorithm to compute the proto-essential subgroups of a $p$-group $S$ without requiring all the subgroups of $S/Z(S)$. This will make the process of finding essential subgroups for a particular group tractable for relatively large $p$-groups. We highlight that the Parker-Semeraro package cannot typically be used to work with groups of order $\geq 3^{10}$, but our package can be used when looking at a Sylow $3$-subgroup of the Monster group $\Mn$, of order $3^{20}$. Appendix \ref{sec:alg} gives a brief description of the version of the algorithm we make use of here. It also lists how it performs with the four $3$-groups we consider in this paper. 

    After computing the proto-essential subgroups of $S$, much of the classification involves understanding the structure of the possible $\calF$-centric radical subgroups in any fusion system $\calF$ on $S$. The normalizer subsystems of these radical subgroups are realized by the $3$-local subgroups in one of the relevant (almost) simple groups. We understand these subsystems by using techniques in local group theory such as coprime action.
    
    To prove that there are no exotic fusion systems supported on each group, we mostly make use of techniques already present in the literature. The work by Onofrei in \cite{ono} allows us to identify parabolic subsystems in a fusion system. Using this information and \cite[Proposition 7.5]{ono}, we deduce that the fusion system is realized by an almost simple group of Lie type in defining characteristic. In other cases, we make use of the work by Oliver in \cite[Proposition 2.11]{todd-modules} to show that two local subsystems are equal.
    
    Because the $3$-groups we are considering are relatively large, we will make use of both GAP \cite{GAP4} and MAGMA \cite{magma} throughout the paper. This code can be found in the GitHub repository \cite{sporadics-code}. If a proof requires some code, we make an explicit reference to the relevant file. We remark that most of the results in this paper can be proven solely using code. Nonetheless, we try to prove most of them theoretically, relying on code as little as we can. This allows us to understand the structure of the fusion systems at a much deeper level and connect them to the $3$-local subgroups of the relevant (almost) simple groups. We further remark that this work serves as a reference to the $3$-local structure of the large sporadic groups, and so complements existing literature.

    The notation we use is mostly standard. We write maps on the right. Groups of Lie type are referred to by their name given in the ATLAS \cite{atlas}. We mainly follow the notations in \cite{ks06} for group theory, and \cite{cra} or \cite{ako} for fusion systems.

    \section{Background} 
    In this section, we recall relevant results from group theory and fusion systems that we shall make use of in this paper. We also establish some further results that we need. 
    
    \subsection{Group Theory.} Unless otherwise specified, we let $G$ be a finite group and $S$ be a finite $p$-group.
    
    \begin{lemma}[Three Subgroups Lemma]
        Let $X, Y, Z \leq G$ be such that $[X,Y,Z] = 1$ and $[Y,Z,X]=1$. Then $[Z,X,Y]=1$.
    \end{lemma}
    \begin{proof}
        This is proven in \cite[Theorem 1.5.6]{ks06}.
    \end{proof}

    \begin{definition}
        We define the \emph{Thompson subgroup} $J(S)$ of $S$ to be the subgroup generated by elementary abelian subgroups of $S$ of largest order.
    \end{definition}

    \begin{lemma} \label{lm:grp-thompson}
        The following hold:
        \begin{enumerate}
            \item $J(S)$ is a characteristic subgroup of $S$.
            \item If $J(S) \leq P \leq S$, then $J(P) = J(S)$.
        \end{enumerate}
    \end{lemma}
    \begin{proof}
        This is \cite[9.2.8]{ks06}.
    \end{proof}

    We shall make use of the following result throughout the paper without explicit reference.
    \begin{lemma}[Coprime Action]
        Let $G$ act coprimely on a finite group $A$, and let $B$ be a $G$-invariant subgroup of $A$. Then the following hold:
        \begin{enumerate}
            \item $A = [A, G]C_A(G)$, and if $A$ is abelian then $A = [A, G] \times C_A(G)$;
            \item if $G$ acts trivially on $A/B$ and $B$, then $G$ acts trivially on $A$; and
            \item if $G$ acts trivially on $A/\Phi(A)$, then $G$ acts trivially on $A$.
        \end{enumerate}
    \end{lemma}
    \begin{proof}
        See \cite[Section 8.2]{ks06}.
    \end{proof}

    We can strengthen (3) from the previous lemma by the following result of Burnside.
    
    \begin{lemma}[Burnside] \label{lm:burnside}
        The subgroup $C_{\Aut(S)}(S/\Phi(S))$ is a normal $p$-subgroup of $\Aut(S)$.
    \end{lemma}
    \begin{proof}
        See \cite[Theorem I.5.1.4]{gor07}.
    \end{proof}

    \begin{lemma} \label{lm:centric-faithful}
        Let $E$ be a finite $p$-group, $A \normalIn E$ and let $G$ be a subgroup of $\Aut(E)$ such that $A$ is $G$-invariant. If $C_E(A) \leq A$, then $C_G(A) \leq O_p(G)$.
    \end{lemma}
    \begin{proof}
        Let $r \in G$ be a $p'$-element that centralizes $A$. Then $[r, A] = 1$, so that $[r, A, E] = 1$. Since $A \normalIn E$, we also find that $[A, E, r] \leq [A, r] = 1$. By the three subgroups lemma, we deduce that $[E, r, A] = 1$. In that case, $[E, r] \leq C_E(A)$. By assumption, we have that $C_E(A) \leq A$. It follows that $r$ centralizes $E/A$. Thus, $r$ centralizes $E$ by coprime action. But $r$ is an automorphism of $E$, so we conclude that $r = 1$.

        Now consider $K := C_G(A)$. Since $A$ is $G$-invariant, we know that $K$ is a normal subgroup of $G$. Moreover, we showed above that if $r \in K$ is a $p'$-element, then $r = 1$. We deduce that $K \leq O_p(G)$, as desired.
    \end{proof}

    Let $A \leq H \leq G$. We say that $A$ is \emph{weakly closed} in $H$ with respect to $G$ if for all $x \in G$ with $A^x \leq H$, we have that $A^x = A$.
    \begin{lemma} \label{lm:weak-closure-equality}
        Let $X, Y, Z \leq G$ be such that:
        \begin{enumerate}
            \item $X$ is a subgroup of $S \in \Syl_p(Y) \cap \Syl_p(Z)$;
            \item $X$ is weakly closed in $S$ with respect to $G$;
            \item there exists a $g \in G$ such that $Y^g \leq Z$; 
            \item $N_Y(X) = N_Z(X)$; and
            \item $N_G(N_Z(X)) \leq N_G(Z)$.
        \end{enumerate}
        Then $Y \leq Z$. In particular, if $Y^g = Z$, then $Y = Z$.
    \end{lemma}
    \begin{proof}
        Let $g$ be as in (3). Then we find that $N_Y(X)^g = N_{Y^g}(X^g) \leq N_Z(X^g)$. As $X^g \leq Y^g \leq Z$, there exists a Sylow $p$-subgroup $T$ of $Z$ such that $X^g \leq T$. Since $S \in \Syl_p(Z)$ as well, there exists an $h \in Z$ such that $T^h = S$. This means that $X^{gh} \leq T^h = S$. As $X$ is weakly closed in $S$ with respect to $G$, we infer that $X^{gh} = X$. Thus,
        \[N_Y(X)^{gh} \leq N_Z(X^g)^h = N_Z(X) = N_Y(X).\]
        In particular, we have
        \[gh \in N_G(N_Z(X)) \leq N_G(Z)\]
        by (5). We deduce that $Z^{gh} = Z$. But then $Y^{gh} \leq Z^h = Z = Z^{gh}$, allowing us to conclude that $Y \leq Z$.
    \end{proof}

    \begin{lemma} \label{lm:autf0s-rad-aut}
        Let $A, B \leq G$. If $S \in \Syl_p(A \cap B)$, then
        \[N_A(S) N_B(S) = N_{AB}(S).\]
    \end{lemma}
    \begin{proof}
        Let $g \in N_{AB}(S)$. Then we can write $g = ab$ for some $a \in A$ and $b \in B$. By assumption, we have $S^a = S^{b^{-1}}$. As $S$ is a subgroup of $A$ and $B$, we infer that $S^a \leq A \cap B$. In particular, both $S$ and $S^a$ are Sylow $p$-subgroups of $A \cap B$. By Sylow's Theorems, there exists an $x \in A \cap B$ such that $S^{ax} = S$. Thus,
        \[g = ab = ax \cdot x^{-1}b \in N_A(S) N_B(S).\]
        The inverse containment follows trivially.
    \end{proof}

    \subsection{Fusion Systems.} We refer the reader to \cite{ako} or \cite{cra} for a background on fusion systems. Unless specified otherwise, let $S$ be a finite $p$-group. If $\calF$ is a fusion system on $S$, then for any subgroup $A \leq S$, we denote by $A^\calF$ the set of $\calF$-conjugates of $A$.

    \begin{definition} \label{def:ess}
        Let $\calF$ be a fusion system on $S$ and $A \leq S$.
        \begin{enumerate}
            \item We say that $A$ is \emph{$\calF$-centric} if for all $B \in A^\calF$, $Z(B) = C_S(B)$.
            \item We say that $A$ is \emph{$\calF$-radical} if $O_p(\Out_\calF(A)) = 1$.
            \item We say that $A$ is \emph{essential in $\calF$} if $A$ is $\calF$-centric, fully $\calF$-normalized, and $\Out_\calF(A)$ contains a strongly $p$-embedded subgroup. 
        \end{enumerate}
    \end{definition}
    \noindent We denote by $\calF^{cr}$ the set of $\calF$-centric radical subgroups, and by $\calE(\calF)$ the set of essential subgroups in $\calF$. We note that all essential subgroups are radical, but typically $\calF^{cr} \neq \calE(\calF) \cup \{S\}$.

    We are interested in the essential subgroups of a fusion system as a consequence of the Alperin-Goldschmidt Theorem.
    \begin{theorem}[Alperin-Goldschmidt]
        Let $\calF$ be a saturated fusion system on $S$. Then 
        \[\calF = \langle \Aut_\calF(S), \Aut_\calF(E) \mid E \in \calE(\calF) \rangle_S.\]
    \end{theorem}
    \begin{proof}
        This is given in \cite[Theorem I.3.5]{ako} and \cite[Proposition 7.25]{cra}.
    \end{proof}

    \begin{lemma} \label{lm:autfe-decomposition}
        Assume that $\calF$ is a saturated fusion system on $S$, and $A \leq S$ is fully $\calF$-normalized. If $\alpha \in N_{\Aut_\calF(A)}(\Aut_S(A))$, then $\alpha$ has an extension $\hat{\alpha} \in \Aut_\calF(N_S(A))$.
    \end{lemma}
    \begin{proof}
        As $A$ is fully $\calF$-normalized, we know that $A$ is $\calF$-receptive by \cite[Theorem 4.21]{cra}. In particular, the map $\alpha$ lifts to 
        \[N_\alpha := \{x \in N_S(A) \mid c_x^\alpha \in \Aut_S(A)\}.\]
        We have chosen $\alpha \in \Aut_\calF(A)$ to normalize $\Aut_S(A)$. This implies that $N_\alpha = N_S(A)$. As such, $\alpha$ has an extension $\hat{\alpha} \colon N_S(A) \to N_S(A)$ in $\calF$.
    \end{proof}

    \begin{theorem}[Model Theorem]
        Let $\calF$ be a saturated fusion system on $S$. Let $Q \normalIn S$ be such that $Q$ is both normal and centric in $\calF$. Then the following hold:
        \begin{enumerate}
            \item There exists a model for $\calF$. That is, there exists a group $G$ with $S \in \Syl_p(G)$ such that $\calF = \calF_S(G)$ and $F^*(G) = O_p(G)$.
            \item If $G_1$ and $G_2$ are models for $\calF$, then there exists a group isomorphism $\phi \colon G_1 \to G_2$ such that $\phi|_S = 1$.
            \item If $G$ is a finite group such that $Q = F^*(G)$ and $\Aut_\calF(Q) = \Aut_G(Q)$, then there exists a $\beta \in \Aut(S)$ such that $\beta|_Q = 1$ and $\calF^\beta = \calF_S(G)$. That is, there is a model for $\calF$ that is isomorphic to $G$.
        \end{enumerate}
    \end{theorem}
    \begin{proof}
        See \cite[Theorem I.4.9]{ako}.
    \end{proof}
    \noindent We say that a fusion system $\calF$ on $S$ is \emph{constrained} if $C_S(O_p(\calF)) \leq O_p(\calF)$. The result above states that any constrained fusion system has a model. In particular, such fusion systems cannot be exotic.
    
    For a fusion system $\calF$ on $S$, we define the \emph{normalizer subsystem} $N_\calF(A)$ to be the largest subsystem of $\calF$ in which $A$ is normal. 

    \begin{lemma} \label{lm:nfa-sat}
        If $A$ is fully $\calF$-normalized and $\calF$ is saturated, then $N_\calF(A)$ is also saturated. 
    \end{lemma}
    \begin{proof}
        See \cite[Theorem I.5.5]{ako} or \cite[Theorem 4.27]{cra}.
    \end{proof}

    If $\calF$ is a fusion system on $S$, we say that a subgroup $A \leq S$ is \emph{weakly $\calF$-closed} if for every $\phi \in \Hom_\calF(A, S)$, we have $A\phi = A$. We say that $A \leq S$ is \emph{strongly $\calF$-closed} if every $\phi \colon P \to S$ in $\calF$ with $P \leq A$ satisfies $P\phi \leq A$.
    \begin{lemma} \label{lm:weak-closure-to-normality}
        Let $\calF$ be a saturated fusion system on a finite $p$-group $S$, and let $A \leq S$. Then $A \normalIn \calF$ if and only if $A$ is weakly closed in $\calF$ and for every $E \in \calE(\calF)$, we have that $A \leq E$.
    \end{lemma}
    \begin{proof}
        This is a restatement of \cite[Lemma I.4.5]{ako}.
    \end{proof}

    \begin{lemma} \label{lm:thom-weak-closed}
        Let $\calF$ be a fusion system on $S$. Then $J(S)$ is weakly closed in $\calF$.
    \end{lemma}
    \begin{proof}
        Let $\varphi \colon J(S) \to S$ be a map in $\calF$. We recall that $J(S)$ is generated by elementary abelian subgroups of maximal order. Take $E \leq J(S)$ to be an elementary abelian subgroup of maximal order. Then $E \varphi$ must also be an elementary abelian subgroup of maximal order. As such, we have $E \varphi \leq J(S)$. We conclude that the image of $\varphi$ lies in $J(S)$. Thus, $J(S)$ is weakly closed in $\calF$.
    \end{proof}

    We define the \emph{hyperfocal} and the \emph{focal} subgroup of a fusion system $\calF$ on $S$ as follows:
    \begin{align*}
        \hyp(\calF) &:= \langle [A, \phi] \mid A \leq S, \phi \in O^{p}(\Aut_\calF(A)) \rangle \\
        \foc(\calF) &:= \langle [A, \phi] \mid A \leq S, \phi \in \Aut_\calF(A) \rangle = S' \hyp(\calF).
    \end{align*}
    By \cite[Theorem I.7.4]{ako}, there exists a unique saturated subsystem $O^p(\calF)$ of $\calF$ defined on $\hyp(\calF)$ such that for all $A \leq \hyp(\calF)$, we have $O^p(\Aut_\calF(A)) \leq \Aut_{O^p(\calF)}(A)$. By construction, we find that $O^p(\calF) \normalIn \calF$ (see \cite[Section 7.5]{cra}).

    We recall the subsystem $O^{p'}_*(\calF)$ of $\calF$ generated by $O^{p'}(\Aut_\calF(P))$ for all $P \leq S$. This is not necessarily a saturated subsystem of $\calF$. We set
    \[\Aut_\calF^0(S) := \langle \alpha \in \Aut_\calF(S) \mid \alpha|_P \in \Hom_{O^{p'}_*(\calF)}(P,S) \textrm{ for some } P \in \calF^c \rangle.\]
    By \cite[Theorem 7.7]{ako}, there exists a unique saturated subsystem $O^{p'}(\calF)$ of $\calF$ defined on $S$ containing $O^{p'}_*(\calF)$ such that $\Aut_{O^{p'}(\calF)}(S) = \Aut_\calF^0(S)$. We again have that $O^{p'}(\calF) \normalIn \calF$.
    
    The following results will be used to understand the subsystems $O^p(\calF)$ and $O^{p'}(\calF)$ for certain fusion systems $\calF$ on $S$.
    \begin{lemma} \label{lm:or21-opf}
        Let $\calF$ be a saturated fusion system on $S$. Then 
        \[O^{p'}_*(\calF) = \langle O^{p'}(\Aut_\calF(E)) \mid E \in \calE(\calF) \rangle_S.\]
    \end{lemma}
    \begin{proof}
        This follows from \cite[Lemma 2.7]{or-fsg} and the following remark.
    \end{proof}

    \begin{lemma} \label{lm:autf0s-rad}
        Let $\calF$ be a saturated fusion system on $S$. Then
        \[\Aut_\calF^0(S) = \langle \alpha \in \Aut_\calF(S) \mid \alpha|_P \in \Hom_{O^{p'}_*(\calF)}(P, S) \textrm{ for some } P \in \calF^{cr}\rangle.\]
    \end{lemma}
    \begin{proof}
        See \cite[Proposition 2.8]{or-fsg}.
    \end{proof}

    \begin{lemma} \label{lm:autf0s-ess}
        Let $\calF$ be a saturated fusion system on $S$. Assume that for all $P \in \calF^{cr}$, $P$ is weakly closed in $O^{p'}_*(\calF)$. Then
        \[\Aut_\calF^0(S) = \langle \alpha \in \Aut_\calF(S) \mid \alpha|_E \in O^{p'}(\Aut_\calF(E)) \textrm{ for some } E \in \calE(\calF)\rangle.\]
    \end{lemma}
    \begin{proof}
        Let
        \[A := \langle \alpha \in \Aut_\calF(S) \mid \alpha|_E \in O^{p'}(\Aut_\calF(E)) \textrm{ for some } E \in \calE(\calF)\rangle,\]
        and take some $\alpha \in \Aut_\calF(S)$ such that $\alpha|_P \in \Hom_{O^{p'}_*(\calF)}(P,S)$ for some $P \in \calF^{cr}$. As $P$ is weakly $O^{p'}_*(\calF)$-closed, we infer that 
        \[\alpha|_P \in N_{\Aut_{O^{p'}_*(\calF)}(P)}(\Aut_S(P)).\]
        Combining Lemmas \ref{lm:autf0s-rad-aut} and \ref{lm:or21-opf}, we can decompose
        \[\alpha|_P = \alpha_1 \circ \dots \circ \alpha_n,\]
        where each $\alpha_i \in N_{O^{p'}(\Aut_\calF(E_i))}(\Aut_S(E_i))$ for some $E_i \in \calE(\calF)$. Applying Lemma \ref{lm:autfe-decomposition}, and noting that each $E_i$ is normal in $S$, we deduce that $\alpha \in A$. The result now follows by applying Lemma \ref{lm:autf0s-rad}.
    \end{proof}

    We recall that if $G$ is a finite group, then $G$ contains a strongly $p$-embedded subgroup if and only if $O^{p'}(G)$ contains a strongly $p$-embedded subgroup.
    \begin{lemma} \label{lm:essentials-in-normalizer}
        Let $\calF$ be a saturated fusion system on $S$, and let $A \leq S$ be weakly $O^{p'}_*(\calF)$-closed. Then $N_\calF(A)$ is saturated, with
        \[\calE(N_\calF(A)) = \{E \in \calE(\calF) \mid A \leq E\}.\]
    \end{lemma}
    \begin{proof}
        Let $\phi \colon A \to S$ be a map in $\calF$. By \cite[Lemma I.7.6]{ako}, we can write $\phi = \psi \alpha$, where $\psi$ lies in $O^{p'}_*(\calF)$ and $\alpha \in \Aut_\calF(S)$. By assumption, we have that $A\psi = A$. This implies that $A$ and $A\phi$ are $\Aut_\calF(S)$-conjugate. Thus, $A$ is fully $\calF$-normalized. Then Lemma \ref{lm:nfa-sat} tells us that $N_\calF(A)$ is saturated. 

        Let $\beta \colon X \to Y$ in $O^{p'}_*(\calF)$ be such that $A \leq X$. As $A\beta = A$, we deduce that $\beta$ lies in $N_\calF(A)$. This implies that
        \[\Hom_{O^{p'}_*(\calF)}(X,Y) \subseteq \Hom_{N_\calF(A)}(X,Y).\]
        In particular, if $E \in \calE(\calF)$ is such that $A \leq E$, then we have $O^{p'}(\Out_\calF(E)) \leq \Out_{N_\calF(A)}(E)$. As $E$ is fully $N_\calF(A)$-normalized, it follows that $E \in \calE(N_\calF(A))$.

        Let $E \in \calE(N_\calF(A))$. Since $A \normalIn N_\calF(A)$, Lemma \ref{lm:weak-closure-to-normality} implies that $A \leq E$. As $E$ is fully $N_\calF(A)$-normalized, we deduce that $E$ is fully $O^{p'}_*(\calF)$-normalized. By \cite[Lemma 2.2]{or-fsg}, it follows that $E$ is fully $\calF$-normalized. As 
        \[O^{p'}(\Out_\calF(E)) \leq \Out_{N_\calF(A)}(E) \leq \Out_\calF(E) \]
        we infer that $E \in \calE(\calF)$, and the result follows.
    \end{proof}

    \begin{lemma} \label{lm:hyp-f}
        Let $\calF$ be a saturated fusion system on $S$, and let $A \leq S$. Assume that $[E, \alpha] \leq A$ for all $\alpha \in O^{p'}(\Aut_\calF(E))$ and $E \in \calE(\calF)$. Then for every $\alpha \colon P \to Q$ in $O^{p'}_*(\calF)$, we have that $[P, \alpha] \leq A$. Moreover, if $A$ is weakly $\calF$-closed, then $A$ is strongly $\calF$-closed.
    \end{lemma}
    \begin{proof}
        Let $\alpha \colon P \to S$ be a map in $O^{p'}_*(\calF)$. By Lemma \ref{lm:or21-opf}, we can decompose
        \[\alpha := \alpha_1 \dots \alpha_l,\]
        with $\alpha_i \in O^{p'}(\Aut_\calF(E_i))$ and $E_i \in \calE(\calF)$, for $1 \leq i \leq l$. We prove the result by induction on $l$. By assumption, the result holds if $l=1$.
        
        Assume now that $l > 1$, and take some $x \in P$. By induction, there exists an $a \in A$ such that $a = [x, \alpha_{1} \dots \alpha_{l-1}] = x^{-1}(x \alpha_1 \dots \alpha_{l-1})$. As such, we have 
        \[xa = x \alpha_1 \dots \alpha_{l-1}.\]
        Since $\alpha_l \in O^{p'}(\Aut_\calF(E_l))$, we find that $xa \in E_l$, with $[xa, \alpha_l] = (xa)^{-1} (xa \alpha_l) = b$, for some $b \in A$. We deduce that
        \[xab = xa \alpha_l.\]
        Combining these two equations, we infer that
        \begin{align*}
            [x, \alpha] &= x^{-1}(x \alpha) \\
            &= x^{-1} (x \alpha_1 \dots \alpha_{l-1} \alpha_l) \\
            &= x^{-1} (xa \alpha_l) \\
            &= x^{-1} (xab) = ab \in A.
        \end{align*}
        We conclude that $[P, \alpha] \leq A$ for every map $\alpha$ in $O^{p'}_*(\calF)$.

        Assume further that $A$ is weakly $\calF$-closed. Take a map $\beta \in \Hom_\calF(X,Y)$ with $X \leq A$. By \cite[Lemma I.7.6]{ako}, we can write $\beta = \psi \varphi$, where $\psi \in \Aut_\calF(S)$ and $\varphi$ lies in $O^{p'}_*(\calF)$. By assumption, $A$ is weakly $\calF$-closed. This means that $A$ is $\psi$-invariant, and so $X\psi \leq A$. Since $\varphi$ lies in $O^{p'}_*(\calF)$, we know that $[X\psi, \varphi] \leq A$. As such, $X \beta = X \psi \varphi \leq A$. We conclude that $A$ is strongly $\calF$-closed.
    \end{proof}

    Lemma \ref{lm:hyp-f} implies that
    \[\langle [P, \alpha] \mid P \leq S, \alpha \in O^{p'}(\Aut_\calF(P)) \rangle \leq A,\]
    where we continue following the notation given in Lemma \ref{lm:hyp-f}. We recall that the focal subgroup of the subsystem $O^{p'}(\calF) = \langle O^{p'}_*(\calF), \Aut_\calF^0(S) \rangle_S$ is defined by
    \[\foc(O^{p'}(\calF)) = \langle [P, \alpha] \mid P \leq S, \alpha \in \Aut_{O^{p'}(\calF)}(P) \rangle.\]
    We further have that $O^{p'}(\Aut_\calF(P)) \leq \Aut_{O^{p'}(\calF)}(P)$ for each $P \leq S$. As such, it is natural to ask whether Lemma \ref{lm:hyp-f} can be extended to show that $\foc(O^{p'}(\calF)) \leq A$. Unfortunately, that is not the case. A counterexample is the fusion system $\calF$ on the $3$-group $S := 3 \wr 3$, labeled $\calF(3^4, 7, 2)$ in \cite[Table A.3.1]{algorithms}. Then the maximal subgroup $A \cong 3^{1+2}_+$ of $S$ contains every $V \in \calE(\calF)$, so that $[V, \alpha] \leq V \leq A$ for all $\alpha \in O^{3'}(\Aut_\calF(V))$. Lemma \ref{lm:hyp-f} tells us that $A$ is strongly $\calF$-closed. But $\calF$ is simple, so we must have that $\foc(O^{3'}(\calF)) = \foc(\calF) = S$. Indeed, there exists an $\alpha \in \Aut_\calF(S)$ satisfying $[S, \alpha] \nleq A$ such that $\alpha|_V \in O^{3'}(\Aut_\calF(V))$.

    Let $S$ be a Sylow $3$-subgroup of $\Fi_{23}$. Then $S$ has a unique extraspecial subgroup $\bfW$ of order $3^{1+8}$ and exponent $3$. To classify all corefree fusion systems $\calF$ on $S$, we need to consider the possible choices for $O^{3'}(\Out_\calF(\bfW))$ when $\bfW$ is radical in $\calF$. The following lemma describes all such choices.

    \begin{lemma} \label{lm:aut-sp8}
        Let $\bfW$ be the extraspecial group of order $3^{1+8}$ and exponent $3$, $G \leq \Out(\bfW)$, $V := \bfW/Z(\bfW)$ and $S \in \Syl_3(G)$. Assume that:
        \begin{enumerate}
            \item $O^{3'}(G) = G$ and $O_3(G) = 1$;
            \item $S \cong 3 \wr 3$; and 
            \item $|C_V(S)| = 3$ and $|C_V(s)| = 3^4$ for some $s \in A$ of order $3$, where $A \leq S$ is the unique subgroup of index $3$ that is elementary abelian.
        \end{enumerate}
        Then $O_2(G) \cong 2^{1+6}_-$ and $G$ is isomorphic to a subgroup of $2^{1+6}_- \ldotp \U_4(2)$. Moreover, $G$ acts irreducibly on $V$.
    \end{lemma}
    \begin{proof}
        All code used for this proof can be found in the file \texttt{fi23/sp83.m}. Since $O^{3'}(G) = G$, we know by \cite{extraspecial-autgrp-sp} that $G$ is isomorphic to a subgroup of $\Sp_8(3)$. Appealing to \cite[Tables 8.48 and 8.49]{maximals}, we find that the maximal subgroups $M$ of $\Sp_8(3)$ satisfying $|C_V(T)| = 3$ for some $T \in \Syl_3(M)$ and $[M : O_3(M)]_3 \geq 3^4$ are:
        \begin{table}[h]
            \centering
            \begin{tabular}{lll}
                $M_1 \cong 2 \times 3^{1+6}_+ : \Sp_6(3)$ & $M_2 \cong 3^{8+3} : (\GL_2(3) \times \Sp_4(3))$ & $M_3 \cong 3^{6+6} : (\GL_3(3) \times \SL_2(3))$\\
                $M_4 \cong 3^{10} : \GL_4(3)$ & $M_5 \cong 2^{1+6}_- \ldotp \U_4(2)$ & $M_6 \cong \SL_2(3^3) \ldotp 3$. 
            \end{tabular}
        \end{table}
        
        \noindent We see that a Sylow $3$-subgroup of $M_3/O_3(M_3)$ is isomorphic to $3 \times 3^{1+2}_+$. As such, $G$ cannot be contained inside $M_3$.

        Assume that $G$ is contained inside the subgroup $M_1$. Then $G$ is isomorphic to a subgroup of $M_1/O_3(M_1) \cong 2 \times \Sp_6(3)$ as $O_3(G) = 1$. Using \cite[Tables 8.28 and 8.29]{maximals} and that $G = O^{3'}(G)$, we find that $G$ must be isomorphic to one of the following subgroups: $\Sp_4(3)$, $\SL_2(3^3) : 3$ or $\SL_2(3) \wr 3$. In each case, we see that $G$ has center of order $2$. As such, $G$ must lie inside $C_L(t)$, for some involution $t \in L := O^{3'}(M_1)$. But for every choice of $T := C_L(t)$, we have $|C_V(S_0)| > 3$, where $S_0 \in \Syl_3(T)$. This contradicts $|C_V(S)| = 3$.

        We next consider the subgroup $M_2$. Again, if $G \leq M_2$, then as $O_3(G) = 1$ and $G = O^{3'}(G)$, we have that $G$ is isomorphic to a subgroup of $\SL_2(3) \times \Sp_4(3)$. Since $S \cong 3 \wr 3$, this implies that $G \cong \Sp_4(3)$. Then we find that $G = (O^{3'}(G))'$ is a subgroup of $T := (O^{3'}(M_2))'$. But we have that $|C_V(S_0)| = 3^2$ for $S_0 \in \Syl_3(T)$, which is a contradiction.

        If $G \leq M_4$, then $G$ is isomorphic to a subgroup of $M_4/O_3(M_4) \cong \GL_4(3)$. Using \cite[Tables 8.8 and 8.9]{maximals}, we deduce again that $G = O^{3'}(G) \cong \Sp_4(3)$. We compute that all subgroups $T$ of $M_4$ isomorphic to $\Sp_4(3)$ satisfy $|C_V(S_0)| = 3^2$ for all $S_0 \in \Syl_3(T)$, a contradiction.

        We next look at the subgroup $M_6 := \SL_2(3^3) \ldotp 3$. This group has a Sylow $3$-subgroup isomorphic to $3 \wr 3$ and satisfies $O^{3'}(M_6) = M_6$. But we find that $|C_V(s)| = 3^3$ for every $s \in A$ of order $3$, where $A \leq S_0$ is the subgroup of index $3$ that is elementary abelian. As such, $G$ cannot be contained inside $M_6$.
        
        This forces $G \leq M_5$. We can then check the three conditions in MAGMA to find that the only choices for $G$ are the subgroups lying between $2^{1+6}_- : (3 \wr 3)$ and $2^{1+6}_- \ldotp \U_4(2)$, and the first statement of the lemma holds. For the second statement, we compute that $G$ acts irreducibly on $V$ in each case.
    \end{proof}

    \section{Fusion Systems on a Sylow $3$-subgroup of $\Fi_{22}$} \label{sec:fi22}
    From the ATLAS \cite{atlas}, we see that 
    \[\O_7(3) \leq \Fi_{22} \leq \TE.\]
    We also find that a Sylow $3$-subgroup of $\O_7(3)$ is a Sylow $3$-subgroup of $\TE$. We further see that $\Fi_{22}$ has two conjugacy classes of subgroups isomorphic to $\O_7(3)$, while $\TE$ has three conjugacy classes of subgroups isomorphic to both $\Fi_{22}$ and $\O_7(3)$. Using \cite[Theorem 2]{wilson-te62}, we find the following $3$-local subgroups in $\TE$.
    \begin{align*}
        M_1 &\cong 3^{1+6}_+ : ((2^{3+6} : 3^2) : 2) \leq \TE \\
        M_2 &\cong 3^5 : \SO_5(3) \leq \O_7(3) \\
        M_{3,i} &\cong 3^{3+3} : \SL_3(3) \leq \O_7(3).
    \end{align*}
    For $1 \leq i \leq 3$, let $H_i$ be a subgroup of $\TE$ containing $M_{3,i}$ isomorphic to $\O_7(3)$. Then the subgroups $H_i$ and $H_j$ are conjugate in $\TE$ if and only if $i=j$. Moreover, $M_{3,i}$ is contained in a $\TE$-conjugate of $H_j$ if and only if $i=j$. On the other hand, each $H_i$ contains a $\TE$-conjugate of $M_2$.

    Fix a Sylow $3$-subgroup $S$ of $G := \TE$. We can then choose $M_j$ and $M_{3,i}$ so that $S \in \Syl_3(M_j) \cap \Syl_3(M_{3,i})$ for $1 \leq j \leq 2$ and $1 \leq i \leq 3$. We now construct further $3$-local subgroups of $G$:
    \begin{align*}
        N_1 := M_1 \cap M_2 &\cong (3^5 : 3^3) : (2 \times \Sym(4)) \\
        N_2 := M_2 \cap M_{3,i} &\cong (3^5 : 3^{1+2}_+) : \GL_2(3) \\
        N_{3,i} := M_1 \cap M_{3,i} &\cong (3^{1+6}_+ : 3) : \GL_2(3).
    \end{align*}

    \begin{notation}
        Using the $3$-local subgroups of $G$ given above, we can define the following subgroups of $S$:
        \begin{itemize}
            \item $\bfW := O_3(M_1)$, 
            \item $\bfV := O_3(M_2)$; 
            \item $\bfT_i := O_3(M_{3,i})$ ($1 \leq i \leq 3$);
            \item $Q := O_3(N_1)$;
            \item $R := O_3(N_2)$; and 
            \item $A_i := O_3(N_{3,i})$ ($1 \leq i \leq 3$).
        \end{itemize}
    \end{notation}
    We can also define the relevant $3$-subgroups of $G$ based on the structure of $S$:
    \begin{itemize}
        \item $\bfW$ is the preimage of $J(S/Z(S))$ in $S$;
        \item $\bfV = J(S)$;
        \item $\bfT_i$ are the unique normal subgroups of $S$ that are special of order $3^6$ and satisfy $|\Phi(\bfT_i)| = 3^3$;
        \item $Q = \langle \bfW, \bfV \rangle$; 
        \item $R = C_S(Z_2(S))$; and
        \item $A_i = C_S(Z(\bfT_i)/Z(S))$.
    \end{itemize}
    
    We can see that $R$ is characteristic in $S$. As $\bfW$ and $\bfV$ are characteristic in $S$, it also follows that $Q$ is also characteristic in $S$. On the other hand, there exists a $\theta \in \Aut(S)$ that acts transitively on the set $\{A_1, A_2, A_3\}$. This map also permutes $\{\bfT_1, \bfT_2, \bfT_3\}$ transitively.
    
    For $1 \leq i, j \leq 3$, we set
    \begin{align*}
        G_i &:= \langle N_{G}(Q), N_{G}(R), N_{G}(A_i) \rangle, \\
        G_{i,j} &:= \langle N_{G}(Q), N_{G}(R), N_{G}(A_i), N_{G}(A_j) \rangle.
    \end{align*}
    Then we have $G_i \cong \O_7(3)$ and $G_{i,j} \cong \Fi_{22}$ (if $i \neq j$). By construction, we see that the normalizers of $Q$ and $R$ coincide in $G_i$, $G_{i,j}$ and $G \cong \TE$. On the other hand, we find that $G_i$ contains precisely one copy of $M_{3,i}$ up to $G$-conjugacy, while $G_{i,j}$ contains both $M_{3,i}$ and $M_{3,j}$. Moreover, as $G_i \cong \O_7(3)$ is a group of Lie type in characteristic $3$, we find that the subgroups $Q$, $R$ and $A_i$ are unipotent radicals of the minimal parabolic subgroups in $\O_7(3)$.
    
    The main difference in the $3$-local subgroups of $G_i$, $G_{i,j}$ and $G$ is the normalizer of $\bfW$. In particular, we have:
    \begin{align*}
        N_{G_i}(\bfW) &\cong 3^{1+6}_+ : ((\SL_2(3) \times \Alt(4)) : 2) \\
        N_{G_{i,j}}(\bfW) &\cong 3^{1+6}_+ : ((2^{3+4} : 3^2) : 2) \\
        N_G(\bfW) &\cong 3^{1+6}_+ : ((2^{3+6} : 3^2) : 2).
    \end{align*}
    We can construct the groups $O^{3'}(\Out_H(\bfW))$, for $H$ one of the three groups above, as subgroups of the group $K := \SL_2(3) \wr 3$. Let $N \normalIn K$ be the base subgroup, and let $C$ be a cyclic group of order $3$ not contained in $N$. Set $N_1, N_2, N_3 \normalIn N$ to be the three distinct copies of $\SL_2(3)$ in $N$ such that $N = N_1 \times N_2 \times N_3$. Fix $S_0$ to be a Sylow $3$-subgroup of $K$ containing a Sylow $3$-subgroup of $N$. We now further conjugate $C$ in $K$ so that it is a subgroup of $S_0$. We let $X := O_2(N_1)$, $Y \leq X$ be a subgroup of order $4$, and $Z := Z(N)$. Set $\hat{N}$ be the diagonal copy of $\SL_2(3)$ inside $N$ fixed by the action of $C$. Then
    \begin{align*}
        O^{3'}(\Out_{G_i}(\bfW)) &\cong \langle \hat{N}, Z, C \rangle \cong \SL_2(3) \times \Alt(4) \\
        O^{3'}(\Out_{G_{i,j}}(\bfW)) &\cong \langle \hat{N}, Y, C \rangle \cong 2^{3+4} : 3^2 \\
        O^{3'}(\Out_{G}(\bfW)) &\cong \langle \hat{N}, X, C \rangle \cong 2^{3+6} : 3^2.
    \end{align*}
        
    \begin{proposition} \label{prp:fi22-g-o73}
        Let $\calF := \calF_S(G_i) \cong \calF_S(\O_7(3))$. Then $\calE(\calF) = \{Q, R, A_i\}$ and $\calF^{cr} = \calE(\calF) \cup \{\bfW, \bfV, \bfT_i, S\}$.
    \end{proposition}
    \begin{proof}
        This follows from Borel-Tits Theorem \cite[Corollary 3.1.6]{gls3}.
    \end{proof}

    \begin{proposition} \label{prp:fi22-g-fi22}
        Let $\calF := \calF_S(G_{i,j}) \cong \calF_S(\Fi_{22})$. Then $\calE(\calF) = \{Q, R, A_i, A_j\}$ and $\calF^{cr} = \calE(\calF) \cup \{\bfW, \bfV, \bfT_i, \bfT_j, S\}$.
    \end{proposition}
    \begin{proof}
        This follows from \cite[Table 1]{fi22-radicals}.
    \end{proof}

    \begin{proposition} \label{prp:fi22-g-te62}
        Let $\calF := \calF_S(G) \cong \calF_S(\TE)$. Then $\calE(\calF) = \{Q, R, A_1^{\Aut(S)}\}$ and $\calF^{cr} = \calE(\calF)  \cup \{\bfW, \bfV, \bfT_1^{\Aut(S)}, S\}$.
    \end{proposition}
    \begin{proof}
        This follows from \cite[Theorem 2]{wilson-te62}.
    \end{proof}

    We now let $\calF$ be a saturated fusion system on $S$. We start by establishing certain properties about $\calF$ in general that will allow us to classify all corefree fusion systems on $S$. We compute $S$ in GAP by finding a Sylow $3$-subgroup of $\O_7(3)$. We also construct $\Out_\calF(\bfW)$ as a subgroup of $\Out(\bfW) \cong \Sp_6(3) : 2$ in MAGMA.
    
    \begin{lemma} \label{lm:essentals-fi22}
        We have $\calE(\calF) \subseteq \{Q, R, A_1, A_2, A_3\}$.
    \end{lemma}
    \begin{proof}
        This is computed using the algorithm for finding all proto-essential subgroups of $S$. Further details are given in Appendix \ref{sec:alg}.
    \end{proof}
    
    The code for Lemmas \ref{lm:fi22-w} to \ref{lm:fi22-t} can be found in the file \texttt{fi22/weak-closure.g}.
    
    \begin{lemma} \label{lm:fi22-w}
        The subgroup $\bfW$ is weakly closed in $\calF$. In particular, $\calE(N_\calF(\bfW)) = \calE(\calF) \cap \{Q, A_1, A_2, A_3\}$.
    \end{lemma}
    \begin{proof}
        We make use of GAP to find that $\bfW$ is the unique extraspecial subgroup of exponent $3$ in $S$ of order $3^{1+6}$. As such, $\bfW$ must be weakly closed in $\calF$. The second statement follows by applying Lemma \ref{lm:essentials-in-normalizer} after noting that precisely $Q$ and $A_i$ contain $\bfW$, for $1 \leq i \leq 3$.
    \end{proof}

    \begin{lemma} \label{lm:fi22-v}
        The subgroup $\bfV$ is weakly closed in $\calF$. In particular, $\calE(N_\calF(\bfV)) = \calE(\calF) \cap \{Q, R\}$.
    \end{lemma}
    \begin{proof}
        As $\bfV = J(S)$, Lemma \ref{lm:thom-weak-closed} tells us that $\bfV$ is weakly closed in $\calF$. The second part follows from Lemma \ref{lm:essentials-in-normalizer} since $\bfV$ is contained in $Q$ and $R$, but not $A_i$.
    \end{proof}
    
    \begin{lemma} \label{lm:fi22-t}
        The subgroup $\bfT_i$ is weakly closed in $O^{3'}_*(\calF)$. Moreover, we have $\calE(N_\calF(\bfT_i)) = \calE(\calF) \cap \{A_i, R\}$.
    \end{lemma}
    \begin{proof}
        By construction, we see that both $R$ and $A_i$ contain $\bfT_i$. We further note that $\bfT_i$ is not contained inside $Q$ or $A_j$ for $j \neq i$. By Lemma \ref{lm:or21-opf}, to show that $\bfT_i$ is weakly closed in $O^{3'}_*(\calF)$, it suffices to show that $\bfT_i$ is invariant under the action of $O^{3'}(\Aut_\calF(R))$ and $O^{3'}(\Aut_\calF(A_i))$. As $\bfT_i = C_{A_i}(Z_2(A_i))$, we find that $\bfT_i$ is invariant under $O^{3'}(\Aut_\calF(A_i))$.
        
        We next show that $\bfT_i$ is invariant under $O^{3'}(\Aut_\calF(R))$. We compute that 
        \[Z_2(S) = Z(R) \leq Z_2(A_i) = Z(\bfT_i) \leq Z_2(R) = Z_3(S).\]
        This implies that $S$ centralizes $Z_2(R)/Z(R)$. Since $O^{3'}(\Aut_\calF(R)) = \langle \Aut_S(R)^{\Aut_\calF(R)} \rangle$, we deduce that $O^{3'}(\Aut_\calF(R))$ must also centralize $Z_2(R)/Z(R)$. As such, the intermediate subgroup $Z(\bfT_i)$ is $O^{3'}(\Aut_\calF(R))$-invariant as well. We conclude that $C_R(Z(\bfT_i)) = \bfT_i$ is also $O^{3'}(\Aut_\calF(R))$-invariant. Thus, $\bfT_i$ is weakly $O^{3'}_*(\calF)$-closed. The second part follows from Lemma \ref{lm:essentials-in-normalizer}.
    \end{proof}
    
    The code for Lemmas \ref{lm:fi22-outfa} to \ref{lm:fi22-outfr} can be found in the file \texttt{fi22/automizer.g}.

    \begin{lemma} \label{lm:fi22-outfa}
        If $A_i \in \calE(\calF)$, then $A_i/\bfT_i$ is a natural module for $O^{3'}(\Out_\calF(A_i)) \cong \SL_2(3)$.
    \end{lemma}
    \begin{proof}
        We compute that $[A_i : \Phi(A_i)] = 3^3$, with $\Phi(A_i) \leq \bfT_i \leq A_i$ and $[A_i : \bfT_i] = 3^2$. Since both $A_i$ and $\bfT_i$ are normal in $S$, we infer that $S$ acts trivially on $\bfT_i/\Phi(A_i)$ of order $3$. We know that $O^{3'}(\Aut_\calF(A_i))$ normalizes $\bfT_i$ by Lemma \ref{lm:fi22-t}. As such, $O^{3'}(\Aut_\calF(A_i))$ centralizes $\bfT_i/\Phi(A_i)$.

        Set $K := C_{O^{3'}(\Aut_\calF(A_i))}(A_i/\bfT_i)$. We show that $K = \Inn(A_i)$. Let $r \in K$ be of $3'$ order. We have found above that $r$ acts trivially on $\bfT_i/\Phi(A_i)$. By coprime action, we infer that $r$ acts trivially on $A_i/\Phi(A_i)$. Applying coprime action again, it follows that $r$ acts trivially on $A_i$, so that $r = 1$. This implies that $K$ is a $3$-group. Since $K$ is normal in $\Aut_\calF(A_i))$ and $A_i \in \calE(\calF)$, we deduce that $K = \Inn(A_i)$. Thus, $O^{3'}(\Out_\calF(A_i))$ acts faithfully on $A_i/\bfT_i$ of order $3^2$. We conclude that $A_i/\bfT_i$ is a natural module for $O^{3'}(\Out_\calF(A_i)) \cong \SL_2(3)$.
    \end{proof}

    \begin{lemma} \label{lm:fi22-outfq}
        If $Q \in \calE(\calF)$, then both $\Phi(Q)/Z(S)$ and $Q/\bfV$ are natural modules for $O^{3'}(\Out_\calF(Q)) \cong \Alt(4)$.
    \end{lemma}
    \begin{proof}
        Let $K := C_{O^{3'}(\Aut_\calF(Q))}(\Phi(Q)/Z(S))$, and take some $r \in K$ to be a $3'$-element. Since $[\bfV : \Phi(Q)] = 3$, we know that $S$ centralizes $\bfV/\Phi(Q)$. As such, $r$ must centralize $\bfV/\Phi(Q)$ as well. By coprime action, we find that $r$ centralizes $\bfV/Z(S)$. But as $S$ centralizes $Z(S) = Z(Q)$, we deduce that $r$ centralizes $Z(S)$. By coprime action, we infer that $r$ centralizes $\bfV$. Applying Lemma \ref{lm:centric-faithful}, it follows that $K \leq O_3(\Aut_\calF(Q)) = \Inn(Q)$. Thus, $O^{3'}(\Out_\calF(Q))$ acts faithfully on $\Phi(Q)/Z(S)$ of order $3^3$. In other words, $O^{3'}(\Out_\calF(Q))$ is isomorphic to a subgroup of $\SL_3(3)$.
        
        Looking at the maximal subgroups of $\SL_3(3)$ in \cite[Tables 8.3 and 8.4]{maximals}, we have three choices for $O^{3'}(\Out_\calF(Q))$, namely $\SL_2(3)$, $\Alt(4)$ and $13 : 3$. We know that $\Out_S(Q) \leq O^{3'}(\Out_{G}(Q)) \cong \Alt(4)$. As such, $\Out_S(Q)$ acts indecomposably on $\Phi(Q)/Z(S)$. Thus, $\Phi(Q)/Z(S)$ is a faithful, indecomposable $\GF(3)$-module for $\Out_\calF(Q)$ of dimension $3$. Since $\Aut(Q)$ has no $13$-element, and $\SL_2(3)$ has no faithful modules satisfying these criteria, we conclude that $\Phi(Q)/Z(S)$ is a natural module for $O^{3'}(\Out_\calF(Q)) \cong \Alt(4)$.
        
        It remains to show that $O^{3'}(\Out_\calF(Q))$ acts faithfully on $Q/\bfV$. Let $r \in C_{O^{3'}(\Out_\calF(Q))}(Q/\bfV)$. As $\Phi(Q) = Z_2(Q)$ is a subgroup of $\bfV$, we find that $[\Phi(Q), Q, r] \leq [Z(Q), r] = 1$ and $[Q, r, \Phi(Q)] \leq [\bfV, \Phi(Q)] = 1$. By the three subgroups lemma, it follows that $[r, \Phi(Q), Q] = 1$ as well. This implies that $[r, \Phi(Q)] \leq Z(S)$. In particular, $r$ centralizes $\Phi(Q)/Z(S)$. By the paragraph above, we deduce that $r \in \Inn(Q)$. In other words, $Q/\bfV$ is a natural module for $O^{3'}(\Out_\calF(Q)) \cong \Alt(4)$.
    \end{proof}
    
    \begin{lemma} \label{lm:fi22-outfr}
        If $R \in \calE(\calF)$, then $R/R'\bfV$ and $Z(R)$ are natural modules for $O^{3'}(\Out_\calF(R)) \cong \SL_2(3)$.
    \end{lemma}
    \begin{proof}               
        We recall that $S/\bfV \cong 3 \wr 3$ and $R/\bfV \cong 3^{1+2}_+$. Then $U := R'\bfV = C_S(S/\bfV)$ is characteristic in $R$ and $S$, with $[R : U] = 3^2$. This implies that $O^{3'}(\Aut_\calF(R))$ centralizes $U/\bfV$ of order $3$.
        
        Now let $r \in C_{O^{3'}(\Out_\calF(R))}(R/U)$ be a $3'$-element. By coprime action, we infer that $r$ centralizes $R/\bfV$. Since $\bfV$ is abelian, we deduce that $[r, R, \bfV] = 1$. Moreover, we compute that $[[R, \bfV] : Z(R)] = 3$, so that $[R, \bfV, r] \leq Z(R)$. By the three subgroups lemma, we conclude that $[\bfV, r, R] \leq Z(R)$. In particular, we find that $[\bfV,r] \leq Z_2(R)$. That is, $r$ acts trivially on $R/Z_2(R)$. But we have that $Z_2(R) \leq \Phi(R)$, meaning that $r$ must centralize $R/\Phi(R)$ as well. By coprime action, we deduce that $C_{O^{3'}(\Aut_\calF(R))}(R/U) = \Inn(R)$. In particular, $O^{3'}(\Out_\calF(R))$ acts faithfully on $R/U$ of order $3^2$. Thus, $R/U$ is a natural module for $O^{3'}(\Out_\calF(R)) \cong \SL_2(3)$.
        
        We compute that $\Phi(R)$ is abelian of order $3^4$ with $C_S(\Phi(R)) = \Phi(R)$. By Lemma \ref{lm:centric-faithful}, $O^{3'}(\Out_\calF(R))$ acts faithfully on $\Phi(R)$. We also note that $\Phi(R) = Z_3(S)$ and $Z(R) = Z_2(S)$. As such, $O^{3'}(\Out_\calF(R))$ centralizes $\Phi(R)/Z(R)$. Thus, $O^{3'}(\Out_\calF(R))$ acts faithfully on $Z(R)$ of order $3^2$. We deduce again that $Z(R)$ is a natural $\SL_2(3)$-module.
    \end{proof}
    
    \begin{proposition} \label{prp:fi22-outfw}
        Set $i := |\calE(\calF) \cap \{A_1, A_2, A_3\}|$. If $i \geq 1$, then one of the following holds:
        \begin{enumerate}
            \item $i = 1$, and if $Q \in \calE(\calF)$, then $O^{3'}(\Out_\calF(\bfW)) \cong \SL_2(3) \times \Alt(4)$;
            \item $i = 2$, in which case $Q \in \calE(\calF)$ and $O^{3'}(\Out_\calF(\bfW)) \cong 2^{3+4} : 3^2$; or
            \item $i = 3$, in which case $Q \in \calE(\calF)$ and $O^{3'}(\Out_\calF(\bfW)) \cong 2^{3+6} : 3^2$.
        \end{enumerate}
        In each case, $O^{3'}(\Out_\calF(\bfW))$ acts irreducibly on $\bfW/Z(S)$.
    \end{proposition}
    \begin{proof}          
        All code in this proof can be found in the files \texttt{fi22/sp63.m} and \texttt{fi22/outfw.m}. We recall that $[S : \bfW] = 3^2$. By definition, we have that $\bfW \leq O_3(N_\calF(\bfW))$. By Lemma \ref{lm:fi22-w}, we also know that $\calE(N_\calF(\bfW)) = \calE(\calF) \cap \{Q, A_1, A_2, A_3\}$.
        
        Assume that $|\calE(N_\calF(\bfW))| \geq 2$. Then for $X, Y \in \calE(N_\calF(\bfW))$ distinct, we have that $O_3(N_\calF(\bfW)) \leq X \cap Y = \bfW$ by Lemma \ref{lm:weak-closure-to-normality}. Hence, $\bfW = O_3(N_\calF(\bfW))$.
        
        As $\bfW$ is an extraspecial group of order $3^{1+6}$ and exponent $3$, we know by \cite{extraspecial-autgrp-sp} that $O^{3'}(\Out(\bfW)) \cong \Sp_6(3)$. By considering the subgroups of $O^{3'}(\Out(\bfW))$ in MAGMA, we deduce that $G := O^{3'}(\Out_\calF(\bfW))$ must be isomorphic to one of the three given choices given above (see the file \texttt{fi22/sp63.m}). Further, we infer that $G$ acts irreducibly on $\bfW/Z(S)$. We observe that $X \in \calE(N_\calF(\bfW))$ if and only if 
        \[O_3(N_{\Out_\calF(\bfW)}(\Out_{X}(\bfW))) = \Out_{X}(\bfW),\]
        where $X \in \{Q, A_1, A_2, A_3\}$. Using this information and the value $i$, we make use MAGMA of to compute $\calE(N_\calF(\bfW))$ in each case.
    \end{proof}

    \begin{proposition} \label{prp:fi22-OutFV}
        If $\{Q, R\} \subseteq \calE(\calF)$, then $O_3(N_\calF(\bfV)) = \bfV$ and $O^{3'}(\Out_\calF(\bfV)) \cong \O_5(3)$ acts naturally on $\bfV$.
    \end{proposition}
    \begin{proof}
        All code in this proof can be found in the file \texttt{fi22/sl53.m}. By definition, we have that $\bfV \leq O_3(N_\calF(\bfV))$. Moreover, we know by Lemma \ref{lm:fi22-v} that $\calE(N_\calF(\bfV)) = \{Q, R\}$. By Lemma \ref{lm:fi22-outfq}, we know that $\Out_\calF(Q)$ acts irreducibly on $Q/\bfV$. Since $R \in \calE(\calF)$ as well, we conclude that $O_3(N_\calF(\bfV)) = \bfV$.
        
        Since $\bfV$ is elementary abelian of order $3^5$, we know that $O^{3'}(\Out(\bfV)) \cong \SL_5(3)$. By considering the subgroups of $O^{3'}(\Out(\bfV))$ in MAGMA, we conclude that $O^{3'}(\Out_\calF(\bfV)) \cong \O_5(3)$ acts naturally on the $5$-dimensional module $\bfV$.
    \end{proof}
    
    \begin{proposition} \label{prp:fi22-outft}
        If $\{R, A_i\} \subseteq \calE(\calF)$, then $\bfT_i/\Phi(\bfT_i)$ is a $3$-dimensional faithful $\GF(3)$ module for $O^{3'}(\Out_\calF(\bfT_i)) \cong \SL_3(3)$.
    \end{proposition}
    \begin{proof}
        By Lemma \ref{lm:fi22-t}, we know that $\{R, A_i\} = \calE(N_\calF(\bfT_i))$. Furthermore, Lemma \ref{lm:fi22-outfa} implies that $O^{3'}(\Out_{N_\calF(\bfT_i)}(A_i))$ acts irreducibly on $A_i/\bfT_i$. As such, Lemma \ref{lm:weak-closure-to-normality} tells us that $\bfT_i = O_3(N_\calF(\bfT_i))$. We recall that $\bfT_i$ is a special group of order $3^6$ with $|\Phi(\bfT_i)| = 3^3$. Hence, $O^{3'}(\Out_\calF(\bfT_i))$ is a subgroup of $\SL_3(3)$ which contains $\Out_S(\bfT_i) \cong 3^{1+2}_+$ and has no non-trivial normal $3$-subgroups. We conclude by the Borel-Tits Theorem \cite[Corollary 3.1.6]{gls3} that $O^{3'}(\Out_\calF(\bfT_i)) \cong \SL_3(3)$.
    \end{proof}

    \begin{lemma} \label{lm:fi22-autf0s}
        $\Aut_\calF^0(S)$ normalizes $A_i$ for $1 \leq i \leq 3$.
    \end{lemma}
    \begin{proof}
        The code used in this lemma can be found in \texttt{fi22/autf0s.g}. Using GAP, we compute that
        \[\calF^{cr} \subseteq \{S, Q, R, A_1^{\Aut(S)}, \bfW, \bfV, \bfT_1^{\Aut(S)}\}.\]
        Let $P \in \calF^{cr}$. If $[S : P] \leq 3$, then it must be weakly closed in $O^{3'}_*(\calF)$. Otherwise, we know that $P$ is weakly closed in $O^{3'}_*(\calF)$ by Lemmas \ref{lm:fi22-w}, \ref{lm:fi22-v} and \ref{lm:fi22-t}. Applying Lemma \ref{lm:autf0s-ess}, we infer that
        \[\Aut_\calF^0(S) = \langle \alpha \in \Aut_\calF(S) \mid \alpha|_E \in O^{3'}(\Aut_\calF(E)), E \in \calE(\calF)\rangle.\]
        We consider each choice of $E \in \calE(\calF)$ individually.

        If $E = R$, then we know by Lemma \ref{lm:fi22-t} that $O^{3'}(\Aut_\calF(R))$ normalizes $\bfT_i$ for $1 \leq i \leq 3$. This implies that any $\alpha \in \Aut_\calF(S)$ such that $\alpha|_R \in O^{3'}(\Aut_\calF(R))$ normalizes $A_i = C_S(Z(\bfT_i)/Z(S))$. Instead, if $E = A_j$, then we have $O^{3'}(\Out_\calF(A_j)) \cong \SL_2(3)$. As such, an involution $\alpha \in \Aut_\calF(S)$ such that $\alpha|_{A_j} \in O^{3'}(\Aut_\calF(A_j))$ centralizes $S/A_j$. We found in Lemma \ref{lm:fi22-outfa} that $\alpha$ centralizes $A_j/\bfW$ as well. Thus, $\alpha$ centralizes $S/\bfW$. But then $\alpha$ centralizes the intermediate subgroups $A_i$ for $1 \leq i \leq 3$. Finally, if $E = Q$, then as $O^{3'}(\Aut_\calF(Q)) \cong \Alt(4)$, we know that any $\alpha \in \Aut_\calF(S)$ such that $\alpha|_Q \in O^{3'}(\Aut_\calF(Q))$ lies inside $\Inn(S)$. We conclude that $\Aut_\calF^0(S)$ normalizes $A_i$ for $1 \leq i \leq 3$.
    \end{proof}
    \noindent The result above implies that if $\calF = O^{3'}(\calF)$, then $A_i$ is weakly $\calF$-closed for $1 \leq i \leq 3$. As $Q = \langle \bfW, \bfV \rangle$ and $R = C_S(Z_2(S))$ are maximal subgroups of $S$ that are also characteristic in $S$, $Q$ and $R$ are always weakly $\calF$-closed.
    
    \begin{proposition} \label{prp:fi22-trivcore}
        Let $\calF$ be a saturated fusion system on $S$. Then either $\calF$ is constrained or we have $O_3(\calF) = 1$. Moreover, $O_3(\calF) = 1$ if and only if $\calE(\calF)$ contains $Q$, $R$ and some $A_i$.
    \end{proposition}
    \begin{proof}
        If $\calE(\calF) \cap \{A_1, A_2, A_3\} = \varnothing$, then $\bfV \normalIn \calF$ by Lemma \ref{lm:fi22-v}. As $\bfV$ is self-centralizing in $S$, $\calF$ is constrained in this case. If we instead have $R \not\in \calE(\calF)$, then $\bfW \normalIn \calF$ by Lemma \ref{lm:fi22-w}. As $\bfW$ is self-centralizing in $S$, $\calF$ is constrained again. Finally, if $Q \not\in \calE(\calF)$, then by Proposition \ref{prp:fi22-outfw}, and as $\bfV \not\normalIn \calF$, we find that $|\calE(\calF) \cap \{A_1, A_2, A_3\}| = 1$. Assuming that $A_i \in \calE(\calF)$, we find that $\bfT_i \normalIn \calF$ by Lemma \ref{lm:fi22-t}, which again implies that $\calF$ is constrained. We have shown that either $\calF$ is constrained or $\calE(\calF)$ contains $Q$, $R$ and some $A_i$. If $O_3(\calF) = 1$, then $\calF$ cannot be constrained. We deduce that if $O_3(\calF) = 1$, then $\calE(\calF)$ contains $Q$, $R$ and some $A_i$.
        
        Conversely, assume that $\calE(\calF)$ contains $Q$, $R$ and some $A_i$. Then Proposition \ref{prp:fi22-outfw} implies that $O^{3'}(\Out_\calF(\bfW))$ acts irreducibly on $\bfW/Z(S)$. This means that $O_3(\calF)$ is equal $\bfW$, $Z(S)$ or $1$. Also, Proposition \ref{prp:fi22-OutFV} implies that $O^{3'}(\Out_\calF(\bfV))$ acts irreducibly on $\bfV$. In particular, we can only have $O_3(\calF) \in \{\bfV, 1\}$. Combining these two results, we conclude that $O_3(\calF) = 1$.
    \end{proof}

    \noindent By the Model Theorem, we know that if $\calF$ is constrained, then $\calF$ has a model. As such, Proposition \ref{prp:fi22-trivcore} states that if $O_3(\calF) \neq 1$, then $\calF$ cannot be exotic. 

    Using all the information we have gathered in this section, we will classify all saturated fusion systems $\calF$ on $S$ such that $O_3(\calF) = 1$. We shall see that there are precisely three choices for $O^{3'}(\calF)$, each realized by a finite simple group. In the process, we shall make use of the work by Onofrei in \cite{ono} on recognizing families of parabolic systems in fusion systems. We remark that it may be possible to use techniques similar to those used in Section \ref{sec:2e62} to complete this proof, thus bypassing \cite{ono}. However, it is harder to infer uniqueness in this case as the choice of $\calE(\calF)$ is not determined by its order. Moreover, this approach allows us to prove the result solely based on the local information about the fusion systems.

    \begin{theorem} \label{thm:fi22-o73}
        Let $\calF$ be a saturated fusion system on $S$ with $O^{3'}(\calF) = \calF$ and $O_3(\calF) = 1$. Fix some $A_i \in \calE(\calF)$. Define the subsystem
        \[\calG := \langle N_\calF(Q), N_\calF(R), N_\calF(A_i) \rangle_S.\]
        Then $\calG$ is isomorphic to $\calF_S(\O_7(3))$.
    \end{theorem}
    \begin{proof}
        We shall prove the result by applying \cite[Theorem 7.5]{ono}. To apply the result, we need to first show that $\calG$ is a \emph{classical parabolic family of fusion systems of type $\calM$}, as described in Definitions 5.1 and 7.4 of \cite{ono}. For ease of notation, we shall assume that $i = 1$. As $O^{3'}(\calF) = \calF$, we infer that $A_1$ is weakly $\calF$-closed by Lemma \ref{lm:fi22-autf0s}. Combining this result with Lemma \ref{lm:fi22-t}, we further deduce that $\bfT_1 = C_S(Z_2(A_1))$ is weakly $\calF$-closed.

        Define $\calB := N_\calG(S)$, $\calG_1 := N_\calG(Q)$, $\calG_2 := N_\calG(R)$ and $\calG_3 := N_\calG(A_1)$. Then each $\calG_j$ is a saturated fusion system on $S$ that has precisely one essential subgroup. Moreover,
        \begin{enumerate}[label=(F\arabic*)]
            \item $\calB$ is a proper subsystem of $\calG_j$ for $1 \leq j \leq 3$ (as $Q$, $R$ and $A_1$ are weakly $\calG$-closed); 
            \item $\calG = \langle \calG_1, \calG_2, \calG_3 \rangle_S$, and $\langle \calG_j, \calG_k \rangle_S$ is not equal to $\calG$ for any $1 \leq j, k \leq 3$; and
            \item $\calG_j \cap \calG_k = \calB$ for $1 \leq j < k \leq 3$ (as $QR = QA_1 = RA_1 = S$).
        \end{enumerate}
        In particular, we satisfy (F1), (F2) and (F3) in \cite[Definition 5.1]{ono}.
        
        We next define the following subsystems of $\calG$:
        \begin{enumerate}
            \item $\calG_{13} := \langle \calG_1, \calG_3 \rangle_S = N_\calG(\bfW)$; 
            \item $\calG_{12} := \langle \calG_1, \calG_2 \rangle_S = N_\calG(\bfV)$; and
            \item $\calG_{23} := \langle \calG_2, \calG_3 \rangle_S = N_\calG(\bfT_1)$.
        \end{enumerate}
        We briefly justify how the equalities above follow. We have $N_\calG(\bfV) = N_\calF(\bfV)$ and $N_\calG(\bfT_1) = N_\calF(\bfT_1)$ by Lemmas \ref{lm:fi22-v} and \ref{lm:fi22-t} respectively. In particular, the fusion systems $N_\calG(\bfV)$ and $N_\calG(\bfT_1)$ are saturated. As $[S : \bfW] = 3^2$ and the subsystems $N_\calF(\bfW)$, $\calG_1 = N_\calF(Q)$ and $\calG_3 = N_\calF(A_1)$ of $\calF$ are saturated, we deduce that $N_\calG(\bfW)$ is saturated by \cite[Theorem A]{bcglo}. By Alperin-Goldschmidt, the equalities above now follow by Lemmas \ref{lm:fi22-w}, \ref{lm:fi22-v} and \ref{lm:fi22-t}. Thus, (F4) holds as well and $\calG$ has a family of parabolic systems.

        We next show that $\calG$ has a family of parabolic systems of type $\calM$. By Lemmas \ref{lm:fi22-outfa}, \ref{lm:fi22-outfq} and \ref{lm:fi22-outfr}, we find that for every $E \in \calE(\calG)$, $O^{3'}(\Out_\calG(E)) = O^{3'}(\Out_\calF(E))$ is a rank one finite group of Lie type in characteristic $3$. As $\calE(\calG) = \{Q, R, A_1\}$, we find that
        \begin{align*}
            O^{3'}(\Out_\calG(\bfW)) &\cong \SL_2(3) \times \Alt(4), \\
            O^{3'}(\Out_\calG(\bfV)) &\cong \O_5(3), \\
            O^{3'}(\Out_\calG(\bfT_1)) &\cong \SL_3(3),
        \end{align*}
        by Propositions \ref{prp:fi22-outfw}, \ref{prp:fi22-OutFV} and \ref{prp:fi22-outft} respectively. In particular, $O^{3'}(\Out_\calF(X))$ is either a group of Lie type in characteristic $3$ of rank two or a (central) product of two groups of Lie type in characteristic $3$ of rank one for $X \in \{\bfV, \bfW, \bfT_1\}$.

        We now consider the structure of the diagram $\calM$. Let $G_{ij}$ be a model for the constrained fusion system $\calF_{ij}$. Then we find that $O^{3'}(G_{12}/O_3(G_{12})) \cong \O_5(3)$ and $O^{3'}(G_{23}/O_3(G_{23})) \cong \SL_3(3)$ are groups of Lie type of rank two. As such, $\calM$ is connected. By \cite[Theorem A]{mei-connected-spherical}, we conclude that $\calM$ is indeed spherical.

        We have confirmed that all the hypotheses of \cite[Theorem 7.5]{ono} are satisfied. We deduce that $\calG$ is realized by a finite group $G$ such that $F^*(G)$ is a group of Lie type in characteristic $3$. Using \cite[Table 3.3.1]{gls3} and that $S$ has $3$-rank $5$, it follows that $F^*(G) \cong \O_7(3)$. Finally, as $O^{3'}(\calG) = \calG$, we conclude that $G \cong \O_7(3)$.
    \end{proof}

    Before looking at the corresponding result for $\Fi_{22}$ and $\TE$, we recall the subgroups $G_i \cong \O_7(3)$ of $G \cong \TE$ satisfying $\calE(\calF_S(G_i)) = \{Q, R, A_i\}$.
    
    \begin{theorem} \label{thm:fi22-fi22}
        Let $\calF$ be a saturated fusion system on $S$ such that $O^{3'}(\calF) = \calF$ and $O_3(\calF) = 1$.
        \begin{enumerate}
            \item If $|\calE(\calF)| = 4$, then $\calF$ is isomorphic to $\calF_S(\Fi_{22})$.
            \item If $|\calE(\calF)| = 5$, then $\calF$ is isomorphic to $\calF_S(\TE)$.
        \end{enumerate}
    \end{theorem}
    \begin{proof}
        The file \texttt{fi22/uniqueness.m} contains all the code used in this proof. We recall that there exists a $\theta \in \Aut(S)$ of order $3$ that acts on $\{ A_1, A_2, A_3 \}$ transitively. As such, we can assume, without loss of generality, that $A_1 \in \calE(\calF)$. Define the subsystem $\calF_0$ of $\calF$ by 
        \[\calF_0 := \langle N_\calF(Q), N_\calF(R), N_\calF(A_1) \rangle_S.\]
        By Theorem \ref{thm:fi22-o73}, we know that $\calF_0$ is isomorphic to $\calF_S(G_1) \cong \calF_S(\O_7(3))$. As such, there exists an $\alpha \in \Aut(S)$ such that $\calF_0^\alpha = \calF_S(G_1)$. Let $\calG := \calF_S(G) \cong \calF_S(\TE)$. We now conjugate $\calF$ by $\alpha$. Note that, as $\alpha \in \Aut(S)$, this gives us an isomorphic fusion system to $\calF$. We have found that $\calF_0 \subseteq \calG$.
        
        We now show that $\Aut_\calF(\bfW) \leq \Aut_\calG(\bfW)$. As $\calF_0$ is a subsystem of both $\calF$ and $\calG$, we find that 
        \[N_\calF(Q) = N_{\calF_0}(Q) = N_\calG(Q).\]
        By the Frattini Argument, we have
        \[\Out_\calG(\bfW) = \langle O^{3'}(\Out_\calG(\bfW)), N_{\Out_\calG(\bfW)}(\Out_S(\bfW)) \rangle.\]
        Moreover, we know by Lemma \ref{lm:autfe-decomposition} that
        \[N_{\Out_\calF(\bfW)}(\Out_S(\bfW)) = N_{\Out_\calG(\bfW)}(\Out_S(\bfW)).\]
        We compute that there is precisely one conjugacy class of subgroups in $\Out(\bfW)$ isomorphic to $O^{3'}(\Out_\calG(\bfW))$, where the choices of $O^{3'}(\Out_\calG(\bfW))$ are given in Proposition \ref{prp:fi22-outfw}. In each case, we see that $O^{3'}(\Out_\calF(\bfW))$ is contained inside a maximal subgroup $M \cong \SL_2(3) \wr \Sym(3)$ of $O^{3'}(\Out(\bfW))$. If $M_0$ is the base group of $M$, then we note that $M_0 \cap \Out_Q(\bfW) = 1$ since $O^{3'}(\Out_\calG(Q)) \cong \Alt(4)$. This implies that $\Out_Q(\bfW)$ is the unique maximal subgroup $A$ of $\Out_S(\bfW)$ such that $|C_{\bfW/Z(\bfW)}(A)| = 3^3$. Thus, $\Out_Q(\bfW)$ must be weakly closed in $\Out_S(\bfW)$ with respect to $\Out(\bfW)$. We next compute using MAGMA that
        \[N_{\Out(\bfW)}(N_{\Out_\calG(\bfW)}(\Out_Q(\bfW))) \leq N_{\Out(\bfW)}(\Out_\calG(\bfW)).\]
        As such, Lemma \ref{lm:weak-closure-equality} tells us that $\Aut_\calF(\bfW) \leq \Aut_\calG(\bfW)$. 

        For $1 \leq i, j \leq 3$, we write $G_{i,j} := \langle G_i, G_j \rangle \cong \Fi_{22}$ and $G = G_{1,2,3}$. We set $I := \{1 \leq i \leq 3 \mid A_i \in \calE(\calF)\}$, and define $\calH := \calF_S(G_I)$. We show that $\calF = \calH$. By construction, we see that $\calE(\calF) = \calE(\calH)$. Again, since $\calF_0$ is a subsystem of both $\calF$ and $\calH$, we find that $N_\calF(Q) = N_\calH(Q)$ and $N_\calF(R) = N_\calH(R)$. Moreover, since $N_\calF(\bfV) = \langle N_\calF(Q), N_\calF(R) \rangle_S$, we find that $N_\calF(\bfV) = N_\calG(\bfV)$. 
        
        As $\calF = \langle N_\calF(\bfV), N_\calF(\bfW) \rangle_S$, it suffices to demonstrate that $N_\calF(\bfW) = N_\calH(\bfW)$. We have already shown that $\Aut_\calF(\bfW) \leq \Aut_\calG(\bfW)$. Since $\calE(\calF) = \calE(\calH)$, we can consider the subgroups of $\Aut_\calG(\bfW)$ containing $\Aut_{\calF_0}(\bfW)$ to conclude that $\Aut_\calF(\bfW) = \Aut_\calH(\bfW)$. Finally, we appeal to MAGMA to confirm that the map
        \[H^1(\Out_\calH(\bfW); Z(\bfW)) \to H^1(\Out_{N_\calH(Q)}(\bfW); Z(\bfW))\]
        is surjective. Thus, \cite[Proposition 2.11]{todd-modules} tells us that $N_\calF(\bfW) = N_\calH(\bfW)$. We conclude that $\calF = \calH$, and the result follows.
    \end{proof}

    We now prove the main result.
    \begin{theorem} \label{thm:fi22}
        Let $\calF$ be a saturated fusion system on $S$ with $O_3(\calF) = 1$. Then one of the following holds:
        \begin{enumerate}
            \item $|\calE(\calF)| = 3$, and $\calF$ is realized by $\O_7(3)$ or $\Aut(\O_7(3))$;
            \item $|\calE(\calF)| = 4$, and $\calF$ is realized by $\Fi_{22}$ or $\Aut(\Fi_{22})$; or
            \item $|\calE(\calF)| = 5$, and $\calF$ is realized by $\TE$ or $\TE : 2$;
        \end{enumerate}
    \end{theorem}
    \begin{proof}
        Let $\calF_0 := O^{3'}(\calF)$. By Theorems \ref{thm:fi22-o73} and \ref{thm:fi22-fi22}, we infer that $\calF_0$ is realized by a finite simple group $G$ isomorphic to one of: $\O_7(3)$, $\Fi_{22}$ or $\TE$. Then \cite[Theorem A]{or-fsg} tells us that $\calF_0$ is simple. Using the notation of \cite{bmor}, we deduce that $\red(\calF) = \calF_0$. Appealing to \cite[Corollary D]{bmor}, it follows that $\calF$ is realizable. Moreover, \cite[Theorem C]{bmor} implies that $\calF$ is realized by some subgroup $G \leq H \leq \Aut(G)$ such that $O^{3'}(H) = G$.
    \end{proof}
    Table \ref{tbl1} summarizes the structure of the automizers in the six corefree fusion systems on $S$.

    \begin{table} \label{tbl1}
        \centering 
        \begin{tabular}{c|c|c|c|c|c|c}
            $\Out_\calF(Q)$ & $\Out_\calF(R)$ & $\Out_\calF(A_1)$ & $\Out_\calF(A_2)$ & $\Out_\calF(A_3)$ & $\Out_\calF(S)$ & $G$ \\
            \hline
            $2 \times \Sym(4)$ & $\GL_2(3)$ & $\GL_2(3)$ & $-$ & $-$ & $2 \times 2$ & $\O_7(3)$ \\
            $2 \times \Sym(4)$ & $\GL_2(3)$ & $-$ & $\GL_2(3)$ & $\GL_2(3)$ & $2 \times 2$ & $\Fi_{22}$ \\
            $2 \times \Sym(4)$ & $\GL_2(3)$ & $\GL_2(3)$ & $\GL_2(3)$ & $\GL_2(3)$ & $2 \times 2$ & $\TE$ \\
            $2 \times 2 \times \Sym(4)$ & $2 \times \GL_2(3)$ & $2 \times \GL_2(3)$ & $-$ & $-$ & $2 \times 2 \times 2$ & $\O_7(3) : 2$ \\
            $2 \times 2 \times \Sym(4)$ & $2 \times \GL_2(3)$ & $-$ & $\GL_2(3)$ & $A_2 \sim_\calF A_3$ & $2 \times 2 \times 2$ & $\Fi_{22} : 2$ \\
            $2 \times 2 \times \Sym(4)$ & $2 \times \GL_2(3)$ & $2 \times \GL_2(3)$ & $\GL_2(3)$ & $A_2 \sim_\calF A_3$ & $2 \times 2 \times 2$ & $\TE : 2$ 
        \end{tabular}
        \caption{Automizer structure for essential subgroups and $S$ in each corefree fusion system on a Sylow $3$-subgroup $S$ of $\Fi_{22}$.}
    \end{table}
    \noindent The notation $A_2 \sim_\calF A_3$ denotes that $A_2$ and $A_3$ are $\calF$-conjugate.

    We end this section with a remark about saturated subsystems. Consider the group $G_{1,2} \cong \Fi_{22}$. By Alperin-Goldschmidt, we know that the fusion system $\calG_{1,2} := \calF_S(G_{1,2})$ can be written as:
    \[\calG_{1,2} = \langle \Aut_{\calG_{1,2}}(S), \Aut_{\calG_{1,2}}(A_1), \Aut_{\calG_{1,2}}(A_2), \Aut_{\calG_{1,2}}(Q), \Aut_{\calG_{1,2}}(R) \rangle_S.\]
    This fusion system has a subsystem given by:
    \[\calG_{1} := \langle \Aut_{\calG_{1,2}}(S), \Aut_{\calG_{1,2}}(A_1), \Aut_{\calG_{1,2}}(Q), \Aut_{\calG_{1,2}}(R) \rangle_S.\]
    We have shown that this is a saturated subsystem of $\calG_{1,2}$ realized by $\O_7(3)$. The following is also a subsystem of $\calG_{1,2}$:
    \[\calH := \langle \Aut_{\calG_{1,2}}(S), \Aut_{\calG_{1,2}}(A_1), \Aut_{\calG_{1,2}}(A_2), \Aut_{\calG_{1,2}}(R) \rangle_S.\]
    By Theorem \ref{thm:fi22}, this cannot be a saturated fusion system. In fact, if $\calH$ is saturated, then the following normalizer subsystem of $\calH$ has a model:
    \[N_{\calH}(\bfW) = \langle \Aut_{\calG_{1,2}}(S), \Aut_{\calG_{1,2}}(A_1), \Aut_{\calG_{1,2}}(A_2) \rangle_S.\]
    But then $O^{3'}(\Out_{\calH}(\bfW))$ must be listed in Proposition \ref{prp:fi22-outfw}, which is a contradiction. This also explains why $\calH$ cannot support a parabolic system -- condition (F4) of \cite[Definition 5.1]{ono} is not satisfied.

    \section{Fusion Systems on a Sylow $3$-subgroup of $\TE : 3$} \label{sec:2e62}
    Based on the work in the previous section on fusion systems on a Sylow $3$-subgroup of $\TE$, we can readily classify all corefree fusion systems on a Sylow $3$-subgroup of $\TE : 3$. This is also quite an interesting case in general since we have a $3$-extension of $\TE$ which is not as well understood as the simple group itself.
    
    Using the ATLAS \cite{atlas}, we find that the following are some of the $3$-local subgroups of $\TE : 3$.
    \begin{align*}
        M_1 &\cong 3^{1+6}_+ : (2^{3+6} : 3^{1+2}_+ : 2) \leq \TE : 3 \\
        M_2 &\cong 3^6 : \SO_5(3) \leq \TE : 3 \\
        M_{3,i} &\cong 3^{3+3} : \SL_3(3) \leq \O_7(3)
    \end{align*}
    We note that the section $2^{3+6}$ of $M_1$ is isomorphic to $Q_8 \times Q_8 \times Q_8$. 
    
    Let $S$ be a Sylow $3$-subgroup of $G := \TE : 3$. We can then choose $M_j$ and $M_{3,i}$ so that $S \in \Syl_3(M_j)$ and $S \cap M_{3,i} \in \Syl_3(M_{3,i})$ for $1 \leq j \leq 2$ and $1 \leq i \leq 3$. We can construct some further $3$-local subgroups of $G$ as follows:
    \begin{align*}
        N_1 &:= M_1 \cap M_2 = 3^{1+6}_+ : 3^2 : (2 \times \Sym(4)) \\
        N_2 &:= M_2 \cap M_{3,1} = 3^6 : 3^{1+2}_+ : \GL_2(3) \\
        N_{3,i} &:= M_1 \cap M_{3,i} \cong 3^{1+6}_+ : 3 : \GL_2(3).
    \end{align*}
    
    \begin{notation}
        Using the $3$-local subgroups of $G$ given above, we can define the following subgroups of $S$:
        \begin{itemize}
            \item $\bfW := O_3(M_1)$;
            \item $\bfV := O_3(M_2)$;
            \item $Q := O_3(N_1)$;
            \item $R := O_3(N_2)$; and
            \item $X_i := O_3(N_{3,i})$ ($1 \leq i \leq 3$).
        \end{itemize}
    \end{notation}
    
    We now define some further $3$-subgroups of $S$. We know that 
    \[O^{3'}(\Out_G(\bfW)) = 2^{3+6} : 3^{1+2}_+,\]
    with $\Out_S(\bfW) \cong 3^{1+2}_+$. Under the action of $\Aut(S)$, we find that the subgroups $X_1$, $X_2$ and $X_3$ are conjugate to six further subgroups containing $\bfW$ of index $3$. We shall label these subgroups $X_i$ for $4 \leq i \leq 9$. Using the structure of $\Out_G(\bfW)$, we deduce that
    \[N_G(X_i) \cong 3^{1+6}_+ : 3 : \SL_2(3)\]
    for $4 \leq i \leq 9$. We contrast this with the normalizers $N_G(X_i) = N_{3,i}$ above for $1 \leq i \leq 3$. In this case, there exists an involution in $N_G(S)$ that normalizes $X_1^S$ but not $X_4^S$.
    
    We have found 9 of the 13 subgroups of $\Out_S(\bfW) \cong 3^{1+2}_+$ of order $3$. Let $Y$ be a subgroup of $S$ containing $\bfW$ not equal to $X_i$ such that $Y \not\normalIn S$. Then $N_S(Y)$ contains $Y$ but is not equal to $N_S(X_i)$ for $1 \leq i \leq 9$. Since no $X_i$ is a subgroup of $Q$, we must have that $N_S(Y) = Q$. Looking at the structure of $\Out_G(\bfW)$ again, we find that
    \[N_G(Y) \cong 3^{1+6}_+ : 3 : ((2 \times 2 \times \SL_2(3)) : 2).\]
    We note that $|Y^S| = [S : N_S(Y)] = 3$.

    We can also define the $3$-subgroups of $G$ described above as the following subgroups of $S$:
    \begin{itemize}
        \item $\bfW$ is the preimage of $J(S/Z(S))$ in $S$;
        \item $\bfV = J(S)$;
        \item $R = C_S(Z_2(S))$;
        \item $Q = \langle \bfW, \bfV \rangle$;
        \item $X_i$ and $Y$ are precisely the intermediate subgroups $\bfW \leq A \leq S$ of order $3^8$ such that $A \neq C_S(S/\bfW)$; and
        \item $Y$ is a subgroup of $S$ containing $\bfW$ such that $N_S(Y) = Q$, unique up to $S$-conjugacy.
    \end{itemize}
    Let $\calG := \calF_S(G)$. Starting from $S$, we cannot determine for a given $X_i$, whether $X_i^\calG$ has length $3$ or $6$.
    
    Now let $H := O^3(G) \cong \TE$. The subgroup $A := N_S(X_i)$ is a Sylow $3$-subgroup of $H$ for $1 \leq i \leq 3$. Using Section \ref{sec:fi22}, we see that $X_i$ are essential in $\calF_A(H)$ for $1 \leq i \leq 3$. We further see that $R_0 := R \cap A$ and $Q_0 := Q \cap A$ are the two remaining essential subgroups of $\calF_A(H)$. The subgroup $\bfW$ is contained in $A$ and is still radical in $\calF_A(H)$. Similarly, the subgroup $\bfV_0 := \bfV \cap A = J(A)$ is also radical in $\calF_A(H)$. As such, their normalizers in $H$ can also be read from Section \ref{sec:fi22}. On the other hand, a subgroup $T \in \{X_i, Y \mid 4 \leq i \leq 9\}$ is not contained in $A$, and instead satisfies $T \cap A = \bfW$. Moreover, $N_A(T)$ is a Sylow $3$-subgroup of $N_H(T)$, with
    \begin{align*}
        N_H(X_i) &\cong 3^{1+6}_+ : \SL_2(3) \\
        N_H(Y) &\cong 3^{1+6}_+ : ((2 \times 2 \times \SL_2(3)) : 2).
    \end{align*}

    From now, let $\calF$ be a fusion system on $S$. We do not construct $G$, but instead work with sections of $G$. We construct the group $S$ by taking a Sylow $3$-subgroup $S_0$ of $\TE$, and then constructing the semidirect product $S_0 : \langle \alpha \rangle$, where $\alpha \in \Aut(S_0)$ is an element of order $3$ that permutes the subgroups $X_i$ for $1 \leq i \leq 3$. Up to isomorphism, there is a unique group $S$ satisfying $|\Aut(S)|_{3'} = 2^3$. The group $\Out_G(\bfW)$ can be constructed in MAGMA by looking at subgroups of $\Out(\bfW) \cong \Sp_6(3) : 2$. We construct the normalizers $N_G(\bfW)$ and $N_G(\bfV)$ as well in GAP using their structure described above.

    \begin{lemma} \label{lm:2e62-ess}
        We have $\calE(\calF) \subseteq \{Q,R, X_1^{\Aut(S)},Y^S\}$.
    \end{lemma}
    \begin{proof}
        See Appendix \ref{sec:alg}.
    \end{proof}

    The code for Lemmas \ref{lm:2e62-w} and \ref{lm:2e62-v} are given in \texttt{te62/weak-closure.g}.

    \begin{lemma} \label{lm:2e62-w}
        The subgroup $\bfW$ is weakly closed in $\calF$. In particular, we have $\calE(N_\calF(\bfW)) = \calE(\calF) \cap \{Q, X_1^{\Aut(S)}, Y^S\}$.
    \end{lemma}
    \begin{proof}
        We compute in GAP that $\bfW$ is the unique extraspecial subgroup of exponent $3$ in $S$ of order $3^{1+6}$. The second part follows from Lemma \ref{lm:essentials-in-normalizer}.
    \end{proof}

    \begin{lemma} \label{lm:2e62-v}
        The subgroup $\bfV$ is weakly closed in $\calF$. In particular, we have $\calE(N_\calF(\bfV)) = \calE(\calF) \cap \{Q, R\}$.
    \end{lemma}
    \begin{proof}
        As $\bfV = J(S)$, this follows by applying Lemmas \ref{lm:thom-weak-closed} and \ref{lm:essentials-in-normalizer}.
    \end{proof}
    
    \begin{proposition} \label{prp:2e62}
        Let $\calF := \calF_S(G) \cong \calF_S(\TE : 3)$. Then $\calE(\calF) = \{Q,R, X_1^{\Aut(S)}, Y^S\}$.
    \end{proposition}
    \begin{proof}
        We make use of the models of $N_\calF(\bfW)$ and $N_\calG(\bfV)$ in GAP to deduce that $\calE(N_\calF(\bfW)) = \{Q, X_1^{\Aut(S)}, Y^S\}$ and $\calE(N_\calF(\bfV)) = \{Q,R\}$ (see the file \texttt{te62/essentials.g}). The result now follows by Lemmas \ref{lm:2e62-ess}, \ref{lm:2e62-w} and \ref{lm:2e62-v}.
    \end{proof}

    The code for Lemmas \ref{lm:2e62-outfy} to \ref{lm:2e62-outfr} are given in \texttt{te62/automizer.g}.

    \begin{lemma} \label{lm:2e62-outfy}
        If $Y \in \calE(\calF)$, then $Z_2(Y)/\Phi(Y)$ is a natural module for $O^{3'}(\Out_\calF(Y)) \cong \SL_2(3)$.
    \end{lemma}
    \begin{proof}
        We have a chain of subgroups $\Phi(Y) \leq Z_2(Y) \leq \bfW \leq Y$, all of which are normal in $N_S(Y) = Q$. We calculate that $[\bfW : Z_2(Y)] = [Z_2(Y) : \Phi(Y)] = 3^2$. Since $[Y : \bfW] = 3$, we know that $N_S(Y)$ centralizes $Y/\bfW$. This implies that $O^{3'}(\Aut_\calF(Y))$ also centralizes $Y/\bfW$.
        
        Assume now that $r \in O^{3'}(\Out_\calF(Y))$ has $3'$-order and centralizes $Z_2(Y)/\Phi(Y)$. We know that $r$ centralizes $Y/\bfW$, so that 
        \[[Y, r, \bfW] \leq [\bfW, \bfW] = Z(Y).\]
        We compute that $[\Phi(Y), N_S(Y)] = Z(Y)$. In particular, $O^{3'}(\Aut_\calF(Y))$ centralizes $\Phi(Y)/Z(Y)$. Thus,
        \[[\bfW, Y, r] \leq [\Phi(Y), r] \leq Z(Y).\]
        By the three subgroups lemma, it follows that $[r, \bfW, Y] \leq Z(Y)$. We conclude that $[r, \bfW] \leq Z_2(Y)$. Thus, $r$ centralizes $\bfW/Z_2(Y)$. But then $r$ centralizes $Y/\Phi(Y)$ by coprime action, forcing $r = 1$. Since $O_3(\Out_\calF(Y)) = 1$, we deduce that $O^{3'}(\Out_\calF(Y))$ acts faithfully on $Z_2(Y)/\Phi(Y)$ of order $3^2$. In particular, $Z_2(Y)/\Phi(Y)$ is a natural module for $O^{3'}(\Out_\calF(Y)) \cong \SL_2(3)$.
    \end{proof}
    
    \begin{lemma} \label{lm:2e62-outfq}
        If $Q \in \calE(\calF)$, then $Q/\bfV \cong \bfW/\Phi(Q)$ and $\Phi(Q)/Z(S)$ are natural modules for $O^{3'}(\Out_\calF(Q)) \cong \Alt(4)$. Moreover, $V_0 := [S, \bfV] \normalIn N_\calF(Q)$.
    \end{lemma}
    \begin{proof}
        Set $V_0 := [S, \bfV]$, so that $V_0 \normalIn S$. We compute that $[\bfV : V_0] = [V_0 : \Phi(Q)] = 3$ and $[Q : \bfV] = 3^3$. Let $r \in C_{O^{3'}(\Aut_\calF(Q))}(\Phi(Q))$ be of $3'$-order. Then we know that $[r, \Phi(Q)] = 1$, so that $[\Phi(Q), r, Q] = [Q, \Phi(Q), r] = 1$. By the three subgroups lemma, we deduce that $[r, Q] \leq C_Q(\Phi(Q)) = \bfV$. This implies that $[r, Q, \bfV] = [Q, \bfV, r] = 1$. Again, the three subgroups lemma tells us that $[r, \bfV] \leq Z(Q) \leq \Phi(Q)$. We have thus found that $r$ centralizes $Q/\bfV$, $\bfV/\Phi(Q)$ and $\Phi(Q)$. By coprime action, it follows that $r = 1$. Hence, $O^{3'}(\Out_\calF(Q))$ acts faithfully on $\Phi(Q)$. Since $\bfW \leq Q$, we know that $Z(S) = Z(Q)$. In particular, $O^{3'}(\Out_\calF(Q))$ acts faithfully on $\Phi(Q)/Z(S)$ of order $3^3$.

        As $\Out_S(Q) \leq O^{3'}(\Out_G(Q)) \cong \Alt(4)$, we infer that $\Aut_S(Q)$ acts indecomposably on $\Phi(Q)/Z(S)$. We further compute that $\Aut(Q)$ contains no element of order $13$. By considering the indecomposable $\GF(3)$-modules of dimension $3$, we deduce that $O^{3'}(\Out_\calF(Q)) \cong \Alt(4)$.

        We next note that $A := O_2(O^{3'}(\Out_\calF(Q)))$ acts trivially on $\bfV/\Phi(Q)$. By coprime action, we deduce that $A$ must act faithfully on $Q/\bfV$. Since $[Q : \bfV] = 3^3$, we infer that $Q/\bfV$ is a natural module for $O^{3'}(\Out_\calF(Q))$. As $Q = \langle \bfV, \bfW \rangle$ and $\bfV \cap \bfW = \Phi(Q)$, it follows that $Q/\bfV \cong \bfW/\Phi(Q)$ is also a natural $\Alt(4)$-module for $O^{3'}(\Out_\calF(Q))$. Finally, we have that $\Phi(Q) \leq V_0 \leq \bfV$, where $A$ acts trivially on $\bfV/\Phi(Q)$. Since $V_0 = [S, \bfV]$ is characteristic in $S$, we deduce that $V_0 \normalIn N_\calF(Q)$, completing the proof.
    \end{proof}
    
    \begin{lemma} \label{lm:2e62-outfr}
        If $R \in \calE(\calF)$, then both $R/C_R(Z_2(R))$ and $Z(R)$ are natural modules for $O^{3'}(\Out_\calF(R)) \cong \SL_2(3)$.
    \end{lemma}
    \begin{proof}
        As $[C_R(Z_2(R)) : \Phi(R)] = 3$, we know that $O^{3'}(\Aut_\calF(R))$ centralizes $C_R(Z_2(R))/\Phi(R)$. By coprime action, we deduce that $O^{3'}(\Out_\calF(R))$ must act faithfully on $R/C_R(Z_2(R))$ of order $3^2$. In particular, $R/C_R(Z_2(R))$ is a natural module for $O^{3'}(\Out_\calF(R)) \cong \SL_2(3)$.

        Assume now that $r \in O^{3'}(\Aut_\calF(R))$ is a $3'$-element that centralizes $Z(R)$. Since $Z_2(R) = Z_3(S)$ and $Z(R) = Z_2(S)$, we know that $S$, and so $r$, centralizes $Z_2(R)/Z(R)$. In that case, $r$ centralizes $Z_2(R)$ by coprime action. In particular, we have $[r, Z_2(R), R] = 1 =[Z_2(R), R, r]$. By the three subgroups lemma, we deduce that $[R, r, Z_2(R)] = 1$. As such, we have that $[R, r] \leq C_R(Z_2(R))$, so that $r$ centralizes $R/C_R(Z_2(R))$. By the previous paragraph, we conclude that $r = 1$. Thus, $Z(R)$ is a natural $\SL_2(3)$-module for $O^{3'}(\Out_\calF(R))$.
    \end{proof}
    
    \begin{lemma} \label{lm:2e62-outfx}
        If $X_i \in \calE(\calF)$, then $O^{3'}(\Out_\calF(X_i)) \cong \SL_2(3)$.
    \end{lemma}
    \begin{proof}
        We recall that $N_S(X_i)$ is isomorphic to a Sylow $3$-subgroup of $\TE$. As such, this result follows from Lemma \ref{lm:fi22-outfa}.
    \end{proof}

    \begin{proposition} \label{prp:2e62-outfw}
        If $\calE(\calF) \cap \{X_1^{\Aut(S)}, Y^S\} \neq \varnothing$, then $O^{3'}(\Out_\calF(\bfW)) \cong 2^{3+6} : 3^{1+2}_+$, and $\calE(N_\calF(\bfW)) = \{Q, X_1^{\Aut(S)}, Y^S\}$. Moreover, $\Out_\calF(\bfW)$ acts irreducibly on $\bfW/Z(S)$.
    \end{proposition}
    \begin{proof}
        All code in this proof can be found in the files \texttt{te62/sp63.m} and \texttt{te62/outfw.m}. We note that the subgroups $Y$ and $X_i$ contain $\bfW$ but are not normal in $S$. Moreover, we find that $[Y : \bfW] = [X_i : \bfW] = 3$. As such, Lemmas \ref{lm:weak-closure-to-normality} and \ref{lm:2e62-w} tell us that $O_3(N_\calF(\bfW)) = \bfW$.
        
        As $\bfW$ is an extraspecial group of order $3^{1+6}$ and exponent $3$, we know by \cite{extraspecial-autgrp-sp} that $O^{3'}(\Out(\bfW)) \cong \Sp_6(3)$. By considering the subgroups of $O^{3'}(\Out(\bfW))$ in MAGMA, we deduce that $O^{3'}(\Out_\calF(\bfW)) \cong 2^{3+6} : 3^{1+2}_+$. Moreover, we compute that $O^{3'}(\Out_\calF(\bfW))$ acts irreducibly on $\bfW/Z(S)$. We can now consider the structure of $N_\calF(\bfW)$ to deduce that $\calE(N_\calF(\bfW)) = \{Q, X_1^{\Aut(S)}, Y^S\}$ (see the file \texttt{te62/outfw.m}).
    \end{proof}

    \begin{proposition} \label{prp:2e62-trivcore}
        Let $\calF$ be a saturated fusion system on $S$. Then either $\calF$ is constrained or $O_3(\calF) = 1$. Moreover, $O_3(\calF) = 1$ if and only if $\calE(\calF) = \{Q, R, X_1^{\Aut(S)}, Y^S\}$.
    \end{proposition}
    \begin{proof}
        If $R \not\in \calE(\calF)$, then we know that $\bfW \normalIn \calF$ by Lemma \ref{lm:2e62-w}. Since $\bfW$ is self-centralizing in $S$, we find that $\calF$ is constrained. On the other hand, if $\calE(\calF) \subseteq \{Q, R\}$, then we know that $\bfV \normalIn \calF$ by Lemma \ref{lm:2e62-v}. Again, $\bfV$ is $S$-centric, meaning that $\calF$ is constrained. We conclude that if $\calF$ is not constrained, then $R \in \calE(\calF)$ and $\{X_1^{\Aut(S)}, Y^S\} \cap \calE(\calF) \neq \varnothing$. By Proposition \ref{prp:2e62-outfw}, we deduce that $\calE(N_\calF(\bfW)) = \{Q, X_1^{\Aut(S)}, Y^S\}$. This also implies that if $O_3(\calF) = 1$, then $\calE(\calF) = \{Q, R, X_1^{\Aut(S)}, Y^S\}$.

        Assume now that $\calE(\calF) = \{Q, R, X_1^{\Aut(S)}, Y^S\}$. Then Proposition \ref{prp:2e62-outfw} implies that $O^{3'}(\Out_\calF(\bfW))$ acts irreducibly on $\bfW/Z(S)$. Since $R \in \calE(\calF)$ does not contain $\bfW$, we infer that $O_3(\calF) \neq \bfW$. But we also know that $O^{3'}(\Out_\calF(R))$ acts irreducibly on $Z_2(S)$. We conclude that $O_3(\calF) = 1$.
    \end{proof}
    
    Consider next the automizer $O^{3'}(\Out_\calF(R)) \cong \SL_2(3)$. By the Extension Axiom, we can lift an involution $t \in O^{3'}(\Aut_\calF(R))$ to $\hat{t} \in \Aut_\calF(S)$. We compute that $\hat{t}$ normalizes some $N_S(X_i)$ and acts regularly on the other two $\Aut(S)$-conjugates of $N_S(X_i)$. We assume that $i=1$ and fix $A := N_S(X_1)$. Then $A$ is a maximal subgroup of $S$ that is $\Aut_\calF(S)$-invariant. We recall that $A$ is isomorphic to a Sylow $3$-subgroup of $\TE$.
    
    \begin{lemma} \label{lm:2e62-autfs}
        If $O_3(\calF) = 1$, then $|\Aut_\calF(S)|_{3'} \geq 2^2$.
    \end{lemma}
    \begin{proof}
        We know by Proposition \ref{prp:2e62-trivcore} that $X_4 \in \calE(\calF)$. Then Lemma \ref{lm:2e62-outfx} tells us that $O^{3'}(\Out_\calF(X_4)) \cong \SL_2(3)$. By the Extension Axiom, we know that there exists an involution $\alpha \in \Aut_\calF(S)$ such that $\alpha|_{X_4} \in N_{O^{3'}(\Aut_\calF(X_4))}(\Aut_S(X_4))$. Similarly, using Lemma \ref{lm:2e62-outfr}, we can find another involution $\beta \in \Aut_\calF(S)$ such that $\beta|_{R} \in N_{O^{3'}(\Aut_\calF(R))}(\Aut_S(R))$. Since $Z(S) = Z(X_4)$, we know that $\alpha$ acts trivially on $Z(S)$. On the other hand, we know by Lemma \ref{lm:2e62-outfr} that $\beta$ acts non-trivially on $Z(S)$. This shows that $\alpha$ and $\beta$ are not $\Aut_\calF(S)$-conjugate. By Sylow's Theorems, we conclude that $|\Aut_\calF(S)|_{3'} \geq 2^2$.
    \end{proof}
    
    Using all the information we have gathered in this section, we will classify all saturated fusion systems $\calF$ on $S$ such that $O_3(\calF) = 1$.
    
    \begin{theorem} \label{thm:2e62-3}
        Let $\calF$ be a saturated fusion system on $S$ such that $O_3(\calF) = 1$. If $|\Aut_\calF(S)|_{3'} = 2^2$, then $\calF$ is realized by $\TE : 3$.
    \end{theorem}
    \begin{proof}   
        The code used in this theorem can be found in \texttt{te62/foc.g} and \texttt{te62/foc.m}.
    
        We first show that $A = \foc(\calF)$. We start by considering the maps in $O^{3'}_*(\calF)$. In particular, we show that $[P, \alpha] \leq A$ for every $\alpha \colon P \to S$ in $O^{3'}_*(\calF)$. By Lemma \ref{lm:hyp-f}, it suffices to show that $[E, O^{3'}(\Aut_\calF(E))] \leq A$ for every $E \in \calE(\calF)$. If $E \in \{Y, X_i\}$, then we know that $[E : \bfW] = 3$. This implies that $[E, O^{3'}(\Aut_\calF(E))] \leq \bfW \leq A$.
        
        Now let $E = Q$. We recall that $O^{3'}(\Out_\calF(Q)) \cong \Alt(4)$. Since $\Alt(4)$ does not have a faithful $2$-dimensional module over $\GF(3)$, we find that $O_2(O^{3'}(\Out_\calF(Q)))$ acts trivially on $Q/\bfW$ of order $3^2$. Also, we have $A \cap Q \normalIn S$, which implies that $[O^{3'}(\Aut_\calF(Q)), Q] \leq A \cap Q \leq A$.
        
        We now assume that $E = R$. By coprime action, we know that
        \[R/\Phi(R) = [R/\Phi(R), O^{3'}(\Aut_\calF(R))] \times C_{R/\Phi(R)}(O^{3'}(\Aut_\calF(R))).\]
        By Lemma \ref{lm:2e62-outfr}, we further have
        \[C_{R/\Phi(R)}(O^{3'}(\Aut_\calF(R))) = C_R(Z_2(R))/\Phi(R).\]
        As $A \cap R \cap C_R(Z_2(R)) = \Phi(R)$ and $\langle A \cap R, C_R(Z_2(R)) \rangle = R$,
        \[[R/\Phi(R), O^{3'}(\Aut_\calF(R))] = (A \cap R)/\Phi(R).\]
        This implies that $[R, O^{3'}(\Aut_\calF(R))] \leq A$ as well. By Lemma \ref{lm:hyp-f}, we conclude that $[P, \alpha] \leq A$ for every $\alpha \colon P \to S$ in $O^{3'}_*(\calF)$.

        We now consider $\Aut_\calF(S)$. Let $\alpha$ and $\beta$ be as given in Lemma \ref{lm:2e62-autfs}. As $\alpha|_{X_4} \in O^{3'}(\Aut_\calF(X_4))$, we find that $[X_4, \alpha] \leq A$. Also, $[A, \alpha] \leq A$ as $A$ is $\Aut_\calF(S)$-invariant. Since $S = AX_4$, we deduce that $[S, \alpha] \leq A$. Similarly, we have $\beta|_{R} \in O^{3'}(\Aut_\calF(R))$, forcing $[S, \beta] \leq A$ as well. This shows that $[S, \Aut_\calF(S)] \leq A$. By \cite[Lemma I.7.6 (b)]{ako}, we know that $\calF = \langle O^{3'}_*(\calF), \Aut_\calF(S) \rangle_S$. As $[R, O^{3'}(\Aut_\calF(R))] = A \cap R$ and $[Q, O^{3'}(\Aut_\calF(Q))] = \bfW$, we infer that
        \[\foc(\calF) = \langle [P, \phi] \mid P \leq S, \phi \in \Aut_\calF(P) \rangle = A.\]
        We recall that $\foc(\calF) = S' \hyp(\calF)$. We have shown that $\hyp(\calF) \leq A$.

        We next show that $\hyp(\calF) = A$. We know from Proposition \ref{prp:2e62-outfw} that $O^3(\Out_\calF(\bfW)) \cong 2^{3+6} : 3^2 : 2$ acts irreducibly on $\bfW/Z(S)$. We deduce that $\bfW \leq \hyp(\calF)$. As $S/\bfW \cong 3^{1+2}_+$, we have $\bfW S' < A$. This forces $\bfW < \hyp(\calF)$. Let $W_0 := C_S(S/\bfW)$. We have found that $W_0 \leq \hyp(\calF)$. But then $\Phi(S) \leq W_0$, which implies that $\hyp(\calF) = \foc(\calF) = A$.
        
        Consider now the fusion system $\calF_0 := O^3(\calF)$ defined on $\hyp(\calF) = A$. Since $O_3(\calF) = 1$, \cite[Proposition 5.47]{cra} tells us that $O_3(\calF_0) = 1$. We know that $A$ is isomorphic to a Sylow $3$-subgroup of $\Fi_{22}$. By Theorem \ref{thm:fi22}, we deduce that $\calF_0$ is realized by $\O_7(3)$, $\Fi_{22}$ or $\TE$. By \cite[Theorem A]{or-fsg}, $\calF_0$ is a simple fusion system.

        Using the notation of \cite{bmor}, we deduce that $\red(\calF) = \calF_0$. As such, \cite[Corollary D]{bmor} tells us that $\calF$ is realizable. Then \cite[Theorem C]{bmor} allows us to conclude that $\calF$ is realized by some subgroup of $\Aut(G)$, where $\calF_S(G) = \calF_0$. We know that $[S : A] = 3$, meaning that $\Out(G)$ must have a $3$-element. This implies that $G \cong \TE$. As $|\Aut_\calF(S)|_{3'} = 2^2$, we conclude that $\calF$ is realized by $\TE : 3$.
    \end{proof}

    \begin{theorem} \label{thm:t262-sym}
        Let $\calF$ be a saturated fusion system on $S$ such that $O_3(\calF) = 1$. If $|\Aut_\calF(S)|_{3'} = 2^3$, then $\calF$ is realized by $\TE : \Sym(3)$.
    \end{theorem}
    \begin{proof}
        The files \texttt{te62/uniqueness.m} and \texttt{te62/uniqueness.g} contain all the code used in this proof. Let $\calG := \calF_S(\TE : \Sym(3))$. We compute that $|\Aut_\calF(S)|_{3'} = |\Aut_\calG(S)|_{3'} = 2^3 = |\Aut(S)|_{3'}$. By Sylow's Theorems, we may conjugate $\calG$ by some $\alpha \in \Aut(S)$ so that $\Aut_\calF(S) = \Aut_\calG(S)$.

        We compute using GAP that
        \[N_{\Out(R)}(N_{\Out_\calG(R)}(\Out_S(R))) \leq \Out_\calG(R) \leq N_{\Out(R)}(\Out_\calG(R)).\]
        We further compute that there is a unique $\Out(R)$-conjugacy class of subgroups isomorphic to $\Out_\calG(R) \cong 2 \times \GL_2(3)$. As such, Lemma \ref{lm:weak-closure-equality} tells us that $\Aut_\calF(R) = \Aut_\calG(R)$.

        Next, we compute using MAGMA that
        \[N_{\Out(\bfW)}(N_{\Out_\calG(\bfW)}(\Out_S(\bfW))) \leq N_{\Out(\bfW)}(\Out_\calG(\bfW)).\]
        Moreover, $\Out(\bfW)$ has a unique conjugacy class of $O^{3'}(\Out_\calG(\bfW)) \cong 2^{3+6} : 3^{1+2}_+$. By Lemma \ref{lm:weak-closure-equality}, we infer that $O^{3'}(\Aut_\calF(\bfW)) = O^{3'}(\Aut_\calG(\bfW))$. By the Frattini Argument and Lemma \ref{lm:autfe-decomposition}, it follows that $\Aut_\calF(\bfW) = \Aut_\calG(\bfW)$. 
        
        We finally compute that $H^1(\Out_{N_\calF(S)}(\bfW); Z(\bfW)) = 0$. Applying \cite[Proposition 2.11]{todd-modules}, it follows that $N_\calF(\bfW) = N_\calG(\bfW)$. By Alperin-Goldschmidt and Lemma \ref{lm:2e62-w}, we conclude that
        \[\calF = \langle \Aut_\calF(S), \Aut_\calF(R), N_\calF(\bfW) \rangle_S = \langle \Aut_\calF(S), \Aut_\calG(R), N_\calG(\bfW) \rangle_S = \calG,\]
        as desired.
    \end{proof}

    If $O_3(\calF) = 1$, then we saw in Lemma \ref{lm:2e62-autfs} that $|\Aut_\calF(S)|_{3'} \geq 2^2$. We dealt with the case when $|\Aut_\calF(S)|_{3'} = 2^2$ in Theorem \ref{thm:2e62-3}, and the case when $|\Aut_\calF(S)|_{3'} = 2^3$ in Theorem \ref{thm:t262-sym}. As $|\Aut(S)|_{3'} = 2^3$, the following result follows.
    \begin{theorem} \label{thm:2e62}
        Let $\calF$ be a saturated fusion system on $S$ with $O_3(\calF) = 1$. Then $\calF$ is realized by $\TE : 3$ or $\Aut(\TE)$.
    \end{theorem}

    Table \ref{tbl2} summarizes the structure of the automizers in the two corefree fusion systems on $S$.
    \begin{table} \label{tbl2}
        \centering
        \resizebox{\textwidth}{!}{\begin{tabular}{c|c|c|c|c|c|c|c}
            $\Out_\calF(Q)$ & $\Out_\calF(R)$ & $\Out_\calF(X_1)$ & $\Out_\calF(X_4)$ & $\Out_\calF(X_7)$ & $\Out_\calF(Y)$ & $\Out_\calF(S)$ & $G$ \\
            \hline
            $2 \times \Sym(4)$ & $\GL_2(3)$ & $\SL_2(3)$ & $X_1 \sim_\calF X_2$ & $\GL_2(3)$ & $(2 \times 2 \times \SL_2(3)) : 2$ & $2 \times 2$ & $\TE : 3$ \\
            $2 \times 2 \times \Sym(4)$ & $2 \times \GL_2(3)$ & $2 \times \SL_2(3)$ & $X_1 \sim_\calF X_2$ & $2 \times \GL_2(3)$ & $\Dih(8) \times \GL_2(3)$ & $2 \times 2 \times 2$ & $\TE : \Sym(3)$ 
        \end{tabular}}
        \caption{Automizer structure for essential subgroups and $S$ in each corefree fusion system on a Sylow $3$-subgroup $S$ of $\TE : 3$.}
    \end{table}

    We conclude this section with a remark. Let $\alpha \in \Aut(S)$ be an element of order $3$ that does not normalize $X_1^S$. Consider the group $\hat{S} := S : \langle \alpha \rangle$. This is similar to the way we constructed a Sylow $3$-subgroup of $\TE : 3$. As such, it is plausible that $\hat{S}$ supports corefree fusion systems. But we compute that every essential subgroup of $\hat{S}$ contains the subgroup $\bfW$, forcing $\bfW \leq O_3(\calF)$ for every fusion system $\calF$ on $\hat{S}$.

    \section{Fusion Systems on a Sylow $3$-subgroup of $\O^+_8(3)$} \label{sec:o83}
    Before we look at fusion systems on a Sylow $3$-subgroup of $\Fi_{23}$ (and $\Bm$), we note that this $3$-group is isomorphic to a Sylow $3$-subgroup of $\Aut(\O^+_8(3))$. Indeed, there are some corefree fusion systems on the group that are realized by subgroups of $\Aut(\O^+_8(3))$. As such, we start by studying all the corefree fusion systems on a Sylow $3$-subgroup of $\O^+_8(3)$. We will make use of this when we look at fusion systems on a Sylow $3$-subgroup of $\Fi_{23}$. We note that $\O^+_8(3)$ is a group of Lie type of rank $4$ in characteristic $3$, so we will make use of \cite{ono} for the classification. This group is also of independent interest because of the additional symmetries of the Dynkin diagram of $\O^+_8(3)$.
    
    Using \cite[Table 1]{k-o83-maxes}, we find that the following 3-local subgroups are maximal in $\O^+_8(3)$:
    \begin{align*}
        M_1 &\cong 3^{1+8}_+ : ((2^{1+6}_- : 3^3) : 2) \\
        M_{2,i} &\cong 3^6 : \L_4(3).
    \end{align*}
    
    Let $S$ be a Sylow $3$-subgroup of $G := \O^+_8(3)$. We can choose $M_1$ and $M_{2,i}$ ($1 \leq i \leq 3$) so that $S \in \Syl_3(M_1) \cap \Syl_3(M_{2,j})$. We construct some further $3$-local subgroups of $G$:
    \begin{align*}
        N_{1,i} := M_{2,i} \cap M_{2,i+1} &\cong 3^{3+6} : \SL_3(3) \\
        N_{2,i} = M_1 \cap M_{2,i} &\cong (3^{1+8}_+ : 3) : ((\SL_2(3) \circ \SL_2(3)): 2) \\
        N_{3,i} = M_1 \cap M_{2,i} \cap M_{2,i+1} &\cong (3^{1+8}_+ : 3^2) : \GL_2(3) \\
        N_4 = M_{2,1} \cap M_{2,2} \cap M_{2,3} &\cong (3^6 : 3^{1+4}_+) : \GL_2(3).
    \end{align*}
    
    \begin{notation}
        Using the $3$-local subgroups of $G$ given above, we can define the following subgroups of $S$:
        \begin{itemize}
            \item $\bfW = O_3(M_1)$;
            \item $\bfV_i := O_3(M_{2,i})$ ($1 \leq i \leq 3$);
            \item $\bfT_i := O_3(N_{1,i})$ ($1 \leq i \leq 3$);
            \item $\bfU_i := O_3(N_{2,i})$ ($1 \leq i \leq 3$);
            \item $X_i := O_3(N_{3,i})$ ($1 \leq i \leq 3$); and
            \item $R := O_3(N_4)$
        \end{itemize}
        For ease of notation, we set $\bfV_0 := \bfV_3$ and $\bfV_4 := \bfV_1$, and so on.
    \end{notation}
    
    We can also define these $3$-subgroups based on the structure of $S$:
    \begin{itemize}
        \item $\bfW$ is the preimage of $J(S/Z(S))$ in $S$;
        \item $\bfV_i$ are the three elementary abelian subgroups in $S$ of maximal order;
        \item $\bfT_i = C_S(D_i)$, where $D_i := \bfV_i \cap \bfV_{i+1}$; 
        \item $X_i = C_S(D_i/Z(S))$;
        \item $\bfU_i = X_i \cap X_{i+1}$; and
        \item $R = C_S(Z_2(S))$.
    \end{itemize}
    We can see that $R$ is characteristic in $S$. On the other hand, the subgroups $X_i$ are maximal in $S$, but not characteristic. In fact, we find that $X_1^{\Aut(S)} = \{X_1, X_2, X_3\}$. We note that there exists a $3$-element $\theta \in \Aut(\O^+_8(3)) \setminus \O^+_8(3)$ that normalizes $S$ and acts transitively on $X_1^{\Aut(S)}$. This element $\theta$ is a triality automorphism of $\O^+_8(3)$. It also permutes the sets $\bfV_1^{\Aut(S)}$, $\bfU_1^{\Aut(S)}$ and $\bfT_1^{\Aut(S)}$ transitively.
    
    As $G \cong \O^+_8(3)$, we know by the Borel-Tits Theorem that the $3$-local subgroups given above are parabolic in $G$. Moreover, the $3$-subgroups $X_i$ and $R$ are unipotent radicals of the minimal parabolic subgroups of $\O^+_8(3)$. 

    Now let $\calF$ be a saturated fusion system on $S$. We compute $S$ in GAP by finding a Sylow $3$-subgroup of $\O_8^+(3)$. We also construct $\Out_\calF(\bfW)$ as a subgroup of $\Out(\bfW) \cong \Sp_8(3) : 2$ in MAGMA.
    \begin{proposition}
        Let $\calF := \calF_S(G) \cong \calF_S(\O^+_8(3))$. Then $\calE(\calF) = \{R, X_1^{\Aut(S)}\}$ and $\calF^{cr} = \calE(\calF) \cup \{\bfW, \bfV_1^{\Aut(S)}, \bfT_1^{\Aut(S)}$, $\bfU_1^{\Aut(S)}, S\}$.
    \end{proposition}
    \begin{proof}
        This follows from Borel-Tits Theorem \cite[Corollary 3.1.6]{gls3}.
    \end{proof}

    \begin{lemma}
        We have $\calE(\calF) \subseteq \{R, X_1^{\Aut(S)}\}$.
    \end{lemma}
    \begin{proof}
        See Appendix \ref{sec:alg}.
    \end{proof}

    \begin{lemma} \label{lm:o83-w}
        The subgroup $\bfW$ is weakly closed in $\calF$. In particular, $\calE(N_\calF(\bfW)) = \calE(\calF) \cap \{X_1^{\Aut(S)}\}$.
    \end{lemma}
    \begin{proof}
        The code for this lemma can be found in the file \texttt{o83/weak-closure-w.g}. We compute in GAP that $\bfW$ is the unique extraspecial subgroup of exponent $3$ in $S$ of order $3^{1+8}$. The second part follows from Lemma \ref{lm:essentials-in-normalizer}.
    \end{proof}

    The code for Lemmas \ref{lm:o83-ti} to \ref{lm:o83-ui} can be found in the file \texttt{o83/weak-closure.g}.
    \begin{lemma} \label{lm:o83-ti}
        The subgroup $\bfT_i$ is weakly closed in $O^{3'}_*(\calF)$ for $1 \leq i \leq 3$. In particular, $\calE(N_\calF(\bfT_i)) = \calE(\calF) \cap \{R, X_i\}$.
    \end{lemma}
    \begin{proof}
        We see by construction that both $R$ and $X_i$ contain $\bfT_i$. We note that $\bfT_i = C_{X_i}(Z_2(X_i))$. This implies that $\bfT_i$ is characteristic in $X_i$. In particular, $\bfT_i$ is invariant under the action of $O^{3'}(\Aut_\calF(X_i))$.

        We next show that $\textbf{T}_i$ is invariant under $O^{3'}(\Aut_\calF(R))$. We note that $Z(R) = Z_2(S)$, and $Z_2(R) = Z_3(S)$. As such, $O^{3'}(\Aut_\calF(R))$ centralizes $Z_2(R)/Z(R)$. In particular, the subgroup $Z(R) \leq Z(\bfT_i) \leq Z_2(R)$ is $O^{3'}(\Aut_\calF(R))$-invariant. This implies that $C_R(Z(\bfT_i)) = \bfT_i$ is also $O^{3'}(\Aut_\calF(R))$-invariant. Applying Lemma \ref{lm:or21-opf}, we conclude that $\bfT_i$ is weakly $O^{3'}_*(\calF)$-closed. The second part follows from Lemma \ref{lm:essentials-in-normalizer}.
    \end{proof}
    
    \begin{lemma} \label{lm:o83-vi}
        The subgroup $\bfV_i$ is weakly closed in $O^{3'}_*(\calF)$ for $1 \leq i \leq 3$. In particular, $\calE(N_\calF(\bfV_i)) = \calE(\calF) \cap \{R, X_i, X_{i+1}\}$.
    \end{lemma}
    \begin{proof}
        We note that $\bfV_i$ is contained inside $R$, $X_i$ and $X_{i+1}$. By Lemma \ref{lm:o83-ti}, we know that $O^{3'}(\Aut_\calF(R))$ normalizes $\bfT_i$ and $\bfT_{i+1}$. We infer that $O^{3'}(\Aut_\calF(R))$ normalizes $\bfT_i \cap \bfT_{i+1}$. Thus, $O^{3'}(\Aut_\calF(R))$ must normalize $\bfV_i = J(\bfT_i \cap \bfT_{i+1})$ as well. 

        We note that $X_i$ contains precisely two $\Aut(S)$-conjugates of $\bfV_i$. As such, if $\Aut_\calF(X_i)$ does not normalize $\bfV_i$, then a subgroup of index two must normalize $\bfV_i$. In particular, $O^{3'}(\Aut_\calF(X_i))$ normalizes $\bfV_i$. Similarly, $O^{3'}(\Aut_\calF(X_{i+1}))$ normalizes $\bfV_i$. The result now follows by applying Lemmas \ref{lm:or21-opf} and \ref{lm:essentials-in-normalizer}.
    \end{proof}
    
    \begin{lemma} \label{lm:o83-ui}
        The subgroup $\bfU_i$ is weakly closed in $O^{3'}_*(\calF)$ for $1 \leq i \leq 3$. In particular, $\calE(N_\calF(\bfU_i)) = \calE(\calF) \cap \{X_i, X_{i+1}\}$.
    \end{lemma}
    \begin{proof}
        Since $S/\bfW$ is elementary abelian, we deduce that $S$ centralizes $X_i/\bfW$. In particular, $O^{3'}(\Aut_\calF(X_i))$ normalizes $\bfW \leq \bfU_i \leq X_i$. Similarly, $O^{3'}(\Aut_\calF(X_{i+1}))$  also normalizes $\bfU_i$. The result now follows by applying Lemmas \ref{lm:or21-opf} and \ref{lm:essentials-in-normalizer}.
    \end{proof}

    The code for Lemmas \ref{lm:o83-outfx} to \ref{lm:o83-outfu} can be found in the file \texttt{o83/automizer.g}.
    \begin{lemma} \label{lm:o83-outfx}
        If $X_i \in \calE(\calF)$, then $X_i/\bfT_i$ and $Z_2(X_i)/Z(S)$ are natural modules for $O^{3'}(\Out_\calF(X_i)) \cong \SL_2(3)$.
    \end{lemma}
    \begin{proof}
        We compute that $[X_i : \Phi(X_i)] = 3^4$ and that $[S, \bfT_i] = \Phi(X_i)$. As such, $O^{3'}(\Out_\calF(X_i))$ centralizes $\bfT_i/\Phi(X_i)$. By coprime action, we deduce that the group $O^{3'}(\Out_\calF(X_i))$ acts faithfully on $X_i/\bfT_i$ of order $3^2$. We infer that  $X_i/\bfT_i$ is a natural module for $O^{3'}(\Out_\calF(X_i)) \cong \SL_2(3)$.

        Assume now that $r \in O^{3'}(\Aut_\calF(X_i))$ is a $3'$-element that centralizes $Z_2(X_i)/Z(S)$ of order $3^2$. As $Z(S) = Z(X_i)$ has order $3$, $r$ must centralize $Z_2(X_i)$ by coprime action. In particular, $[X_i, Z_2(X_i), r] = 1 = [Z_2(X_i), r, X_i]$. By the three subgroups lemma, we infer that $[r, X_i, Z_2(X_i)] = 1$. Thus, $[r, X_i] \leq C_{X_i}(Z_2(X_i)) = \bfT_i$, so that $r$ centralizes $X_i/\bfT_i$. By the previous paragraph, we conclude that $r = 1$. In particular, $Z_2(X_i)/Z(S)$ is a natural module for $O^{3'}(\Out_\calF(X_i))$.
    \end{proof}

    \begin{lemma} \label{lm:o83-outfr}
        If $R \in \calE(\calF)$, then $Z(R)$, $\bfT_i/(\bfT_i \cap \bfT_{i+1})$ and $(\bfT_i \cap \bfT_{i+1})/\Phi(R)$ are natural modules for $O^{3'}(\Out_\calF(R)) \cong \SL_2(3)$.
    \end{lemma}
    \begin{proof}        
        We compute that $\Phi(R) = \bfT_1 \cap \bfT_2 \cap \bfT_3$ is elementary abelian of order $3^5$ with $C_S(\Phi(R)) = \Phi(R)$. By Lemma \ref{lm:centric-faithful}, we infer that $O^{3'}(\Out_\calF(R))$ acts faithfully on $\Phi(R)$. As $\Phi(R) = Z_3(S)$ and $Z(R) = Z_2(S)$, it follows that $O^{3'}(\Out_\calF(R))$ centralizes $\Phi(R)/Z(R)$. As $Z(R)$ has order $3^2$, we deduce that $Z(R)$ is a natural module for $O^{3'}(\Out_\calF(R)) \cong \SL_2(3)$.

        We further note that $[R, S, S] = \Phi(R)$. By \cite[Lemma 3.4]{niles-pushingup}, we conclude that $R/\Phi(R)$ is a direct sum of two natural $\SL_2(3)$-modules. In particular, both $\bfT_i/(\bfT_i \cap \bfT_{i+1})$ and $(\bfT_i \cap \bfT_{i+1})/\Phi(R)$ are natural $\SL_2(3)$-modules.
    \end{proof}

    \begin{lemma} \label{lm:o83-outft}
        If $\{R, X_i\} \subseteq \calE(\calF)$, then $O^{3'}(\Out_\calF(\bfT_i)) \cong \SL_3(3)$ acts irreducibly on $Z(\bfT_i)$, $\bfV_{i-1}/Z(\bfT_i)$ and $\bfV_i/Z(\bfT_i)$.
    \end{lemma}
    \begin{proof}
        Let $K := O^{3'}(\Out_\calF(\bfT_i))$. Since $X_i \in \calE(\calF)$, we know by Lemma \ref{lm:o83-outfx} that $O^{3'}(\Out_\calF(X_i))$ acts irreducibly on $X_i/\bfT_i$. Similarly, as $R \in \calE(\calF)$, we infer by Lemma \ref{lm:o83-outfr} that $O^{3'}(\Out_\calF(R))$ acts irreducibly on $R/\bfT_i$. This implies that $O_3(K) = 1$.
        
        We have also found that $O^{3'}(\Out_\calF(R)) \cong \SL_2(3)$ acts irreducibly on $Z(R) \leq Z(\bfT_i)$, and that $O^{3'}(\Out_\calF(X_i)) \cong \SL_2(3)$ acts irreducibly on $Z(\bfT_i)/Z(S)$. Since $|Z(\bfT_i)| = 3^3$, we infer that $K/C_K(Z(\bfT_i)) \cong \SL_3(3)$ acts irreducibly on $Z(\bfT_i)$. As $C_S(Z(\bfT_i)) = \bfT_i$, we deduce that $C_K(Z(\bfT_i)) \cap \Out_S(\bfT_i) = 1$. As $\Out_S(\bfT_i)$ is a Sylow $3$-subgroup of $K$, it follows that $C_K(Z(\bfT_i)) = O_{3'}(K)$.

        We recall from Lemma \ref{lm:o83-vi} that $\bfV_i$ is weakly closed in $O^{3'}_*(\calF)$, and contained in $\bfT_i$. As such, the subgroup $C_K(\bfV_i/Z(\bfT_i))$ is normal in $K$. We compute that $C_S(\bfV_i/Z(\bfT_i)) = \bfT_i$, meaning that
        \[C_K(\bfV_i/Z(\bfT_i)) \leq O_{3'}(K) = C_K(Z(\bfT_i)).\]
        As $K/C_K(\bfV_i/Z(\bfT_i))$ acts faithfully on $\bfV_i/Z(\bfT_i)$ of order $3^3$, we are forced to have $C_K(\bfV_i/Z(\bfT_i)) = C_K(Z(\bfT_i))$.

        Now let $r \in O_{3'}(K)$. From our analysis above, it follows that $[r, \bfV_i, \bfT_i] = 1 = [\bfV_i, \bfT_i, r]$. By the three subgroups lemma, we infer that $[\bfT_i, r, \bfV_i] = 1$. But then we find that $r$ centralizes $\bfT_i/\bfV_i$ as well. By coprime action, we deduce that $r$ centralizes $\bfT_i$, so that $r = 1$. We conclude that $O_{3'}(K) = 1$. This implies that $O^{3'}(\Out_\calF(\bfT_i)) \cong \SL_3(3)$ acts irreducibly on $Z(\bfT_i)$ and $\bfV_i/Z(\bfT_i)$. We can replace $\bfV_i$ with $\bfV_{i-1}$ above to complete the proof.
    \end{proof}

    \begin{lemma} \label{lm:o83-outfu}
        If $\{X_i, X_{i+1}\} \subseteq \calE(\calF)$, then $O^{3'}(\Out_\calF(\bfU_i)) \cong \O^+_4(3) \cong \SL_2(3) \circ \SL_2(3)$.
    \end{lemma}
    \begin{proof}
        All code used in this proof can be found in \texttt{o83/outfu.m}. We compute that $[\bfU_i : \Phi(\bfU_i)] = 3^5$, and that $\Phi(\bfU_i)$ is elementary abelian of order $3^5$. As $\Phi(\bfU_i)$ is contained in $\bfV_i$ of order $3^6$, we find that $O^{3'}(\Out_\calF(\bfU_i))$ centralizes $\bfV_i/\Phi(\bfU_i)$. By coprime action, we deduce that $O^{3'}(\Out_\calF(\bfU_i))$ acts faithfully on $\bfU_i/\bfV_i$ of order $3^4$. That is, $O^{3'}(\Out_\calF(\bfU_i))$ is a subgroup of $\SL_4(3)$. We appeal to MAGMA to deduce that $O^{3'}(\Out_\calF(\bfU_i))$ is isomorphic to either $\O_4^+(3)$ or $\Alt(6)$. As $\bfU_i$ is not essential in $\calF$, we conclude that $O^{3'}(\Out_\calF(\bfU_i)) \cong \O_4^+(3)$.
    \end{proof}

    \begin{lemma} \label{lm:o83-autf0s}
        $\Aut_\calF^0(S)$ normalizes $X_i$ for $1 \leq i \leq 3$.
    \end{lemma}
    \begin{proof}
        The code used in this lemma can be found in \texttt{o83/autf0s.g}. Using GAP, we compute that
        \[\calF^{cr} \subseteq \{S, R, X_1^{\Aut(S)}, \bfU_1^{\Aut(S)}, \bfT_1^{\Aut(S)}, \bfV_1^{\Aut(S)}, \bfW\}.\]
        In particular, we deduce that any $P \in \calF^{cr}$ is weakly closed in $O^{3'}_*(\calF)$ by Lemmas \ref{lm:o83-ti}, \ref{lm:o83-vi} and \ref{lm:o83-ui}. Applying Lemma \ref{lm:autf0s-ess}, it follows that
        \[\Aut_\calF^0(S) = \langle \alpha \in \Aut_\calF(S) \mid \alpha|_E \in O^{3'}(\Aut_\calF(E)), E \in \calE(\calF)\rangle.\]
        We consider each choice of $E \in \calE(\calF)$ individually.

        If $E = R$, then we know by Lemma \ref{lm:o83-ti} that $O^{3'}(\Aut_\calF(R))$ normalizes $\bfT_i$ for $1 \leq i \leq 3$. This implies that any $\alpha \in \Aut_\calF(S)$ such that $\alpha|_R \in O^{3'}(\Aut_\calF(R))$ normalizes $X_i = C_S(Z(\bfT_i)/Z(S))$ for $1 \leq i \leq 3$. Instead, if $E = X_j$, then $O^{3'}(\Aut_\calF(X_j))$ normalizes $\bfV_{j-1}$ and $\bfV_j$ by Lemma \ref{lm:o83-vi}. Clearly, any $\alpha \in \Aut_\calF(S)$ such that $\alpha|_{X_j} \in O^{3'}(\Aut_\calF(X_j))$ normalizes each $\bfV_i$ for $1 \leq i \leq 3$. In particular, $\alpha$ normalizes $\bfT_i = C_S(\bfV_i \cap \bfV_{i+1})$ and $X_i = C_S(Z(\bfT_i)/Z(S))$. We conclude that $\Aut_\calF^0(S)$ normalizes $X_i$ for $1 \leq i \leq 3$. 
    \end{proof}

    \begin{proposition} \label{prp:o83-trivcore}
        Let $\calF$ be a saturated fusion system on $S$. Then either $\calF$ is constrained or we have $O_3(\calF) = 1$. Moreover, $O_3(\calF) = 1$ if and only if $\calE(\calF) = \{R, X_1^{\Aut(S)}\}$.
    \end{proposition}
    \begin{proof}
        All code in this proof can be found in \texttt{o83/trivcore.g}. Assume that $O_3(\calF) = 1$. Since $\bfW \not\normalIn \calF$, we know by Lemma \ref{lm:o83-w} that $R \in \calE(\calF)$. Also, $\bfV_i \not\normalIn \calF$, forcing $X_{i-1} \in \calE(\calF)$ for $1 \leq i \leq 3$ by Lemma \ref{lm:o83-vi}. We deduce that $\calE(\calF) = \{R, X_1^{\Aut(S)}\}$. Since $\bfV_i$ and $\bfW$ are self-centralizing in $S$, we also infer that if $|\calE(\calF)| \leq 3$, then $\calF$ is constrained.

        Conversely, we assume that $|\calE(\calF)| = 4$. Set $A := O_3(\calF)$, and assume that $A \neq 1$. We know by Lemma \ref{lm:o83-outft} that $O^{3'}(\Out_\calF(\bfT_i))$ acts irreducibly on $Z(\bfT_i)$. This implies that $Z(\bfT_i) \leq A$ for $1 \leq i \leq 3$. Then
        \[Z_3(S) = Z(\bfT_1) Z(\bfT_2) Z(\bfT_3) \leq A \leq X_1 \cap X_2 \cap X_3 = \bfW.\]
        Since $\Phi(A) \leq \Phi(\bfW) = Z(S)$ is normal in $\calF$, we must have that $\Phi(A) = 1$. In other words, $A$ is (elementary) abelian. As $Z_3(S)$ is self-centralizing in $S$, it follows that $A = Z_3(S)$. As such, the subgroup $A$, and so $[A, X_i] = Z_2(S)$, is normal in $N_\calF(X_i)$ as well. But we have $Z(S) < Z_2(S) < Z_2(X_i)$, contradicting Lemma \ref{lm:o83-outfx}.
    \end{proof}
    
    Using all the information we have gathered in this section, we will classify all saturated fusion systems $\calF$ on $S$ such that $O_3(\calF) = 1$. As in Section \ref{sec:fi22}, we shall make use of the work by Onofrei in \cite{ono}. We remark again that, although it may be possible to prove the result without appealing to \cite{ono}, it would be hard to infer uniqueness as not all the essential subgroups are characteristic in $S$ (cf. the remark before Theorem \ref{thm:fi22-o73}).
    
    \begin{theorem} \label{thm:o83}
        Let $\calF$ be a saturated fusion system on $S$ such that $O_3(\calF) = 1$. Then $\calF \cong \calF_S(G)$ with $F^*(G) = O^{3'}(G) \cong \O^+_8(3)$.
    \end{theorem}
    \begin{proof}
        By Proposition \ref{prp:o83-trivcore}, if $O_3(\calF) = 1$, then $\calE(\calF) = \{R, X_1^{\Aut(S)}\}$. We apply \cite[Theorem 7.5]{ono} like in Theorem \ref{thm:fi22-o73}. Let $\calG := O^{3'}(\calF)$. We know by Lemma \ref{lm:o83-autf0s} that each $X_i$ is weakly $\calG$-closed. Combining this result with Lemmas \ref{lm:o83-ti} and \ref{lm:o83-ui}, we deduce that $\bfT_i = C_S(Z_2(X_i))$ and $\bfU_i = X_i \cap X_{i+1}$ are also weakly $\calG$-closed.

        Define $\calB := N_\calG(S)$, $\calG_i := N_\calG(X_i)$ for $1 \leq i \leq 3$ and $\calG_4 := N_\calG(R)$. Then we find that each $\calG_i$ is a saturated fusion system on $S$ that has precisely one essential subgroup. Moreover, $X_i$ and $R$ are weakly $\calG$-closed. Since they are also maximal in $S$, we infer that $\calG_i \cap \calG_j = \calB$ for $1 \leq i < j \leq 4$. We now define $\calG_{ij} := \langle \calG_i, \calG_j \rangle_S$ for $1 \leq i < j \leq 4$. By Lemmas \ref{lm:o83-ti}, \ref{lm:o83-vi} and \ref{lm:o83-ui}, we find that each $\calG_{ij}$ is saturated and constrained. Thus, $\calG$ has a family of parabolic systems.

        We know by Lemmas \ref{lm:o83-outfx} and \ref{lm:o83-outfr} that $O^{3'}(\Out_\calG(E)) \cong \SL_2(3)$ for every $E \in \calE(\calG)$. Moreover, $O^{3'}(\Out_\calG(O_3(\calG_{i,j})))$ is either a group of Lie type of rank $2$ or a central product of two groups of Lie type of rank $1$ by Lemmas \ref{lm:o83-outft} and \ref{lm:o83-outfu}. Thus, $\calG$ has a family of parabolic systems of type $\calM$.

        We now consider the structure of the diagram $\calM$. Let $G_{i,j}$ be a model for $\calG_{i,j}$. Since $O^{3'}(G_{i,4}/O_3(G_{i,4})) \cong \L_3(3)$ is a group of rank two, we find that the diagram $\calM$ is connected. By \cite[Theorem A]{mei-connected-spherical}, we infer that $\calM$ is indeed spherical.

        Appealing to \cite[Theorem 7.5]{ono}, it follows that $\calG \cong \calF_S(G)$, where $G$ is a group of Lie type in characteristic $3$. Using \cite[Table 3.3.1]{gls3} and that $S$ has $3$-rank $6$, we deduce that $G \cong \O^+_8(3)$. Finally, \cite[Theorem C]{bmor} allows us to conclude that $\calF \cong \calF_S(H)$, for $G \leq H \leq \Aut(G)$ and $O^{3'}(H) = G$.
    \end{proof}

    Table \ref{tbl3} summarizes the structure of the automizers in the seven corefree fusion systems on $S$.
    \begin{table} \label{tbl3}
        \centering 
        \resizebox{\textwidth}{!}{\begin{tabular}{c|c|c|c|c|c}
            $\Out_\calF(R)$ & $\Out_\calF(X_1)$ & $\Out_\calF(X_2)$ & $\Out_\calF(X_3)$ & $\Out_\calF(S)$ & $G$ \\
            \hline
            $\GL_2(3)$ & $\GL_2(3)$ & $\GL_2(3)$ & $\GL_2(3)$ & $2 \times 2$ & $\O^+_8(3)$ \\
            $2 \times \GL_2(3)$ & $2 \times \GL_2(3)$ & $2 \times \GL_2(3)$ & $2 \times \GL_2(3)$ & $2 \times 2 \times 2$ & $\O^+_8(3) : 2_1$ \\
            $2 \times \GL_2(3)$ & $\GL_2(3)$ & $X_1 \sim_\calF X_2$ & $2 \times \GL_2(3)$ & $2 \times 2 \times 2$ & $\O^+_8(3) : 2_2$ \\
            $2 \times 2 \times \GL_2(3)$ & $2 \times 2 \times \GL_2(3)$ & $2 \times 2 \times \GL_2(3)$ & $2 \times 2 \times \GL_2(3)$ & $2 \times 2 \times 2 \times 2$ & $\O^+_8(3) : (2 \times 2)_1$ \\
            $2 \times 2 \times \GL_2(3)$ & $2 \times \GL_2(3)$ & $X_1 \sim_\calF X_2$ & $2 \times 2 \times \GL_2(3)$ & $2 \times 2 \times 2 \times 2$ & $\O^+_8(3) : (2 \times 2)_2$ \\
            $4 \times \GL_2(3)$ & $2 \times \GL_2(3)$ & $X_1 \sim_\calF X_2$ & $4 \times \GL_2(3)$ & $2 \times 2 \times 4$ & $\O^+_8(3) : 4$ \\
            $\Dih(8) \times \GL_2(3)$ & $2 \times 2 \times \GL_2(3)$ & $X_1 \sim_\calF X_2$ & $\Dih(8) \times \GL_2(3)$ & $2 \times 2 \times \Dih(8)$ &  $\O^+_8(3) : \Dih(8)$
        \end{tabular}}
        \caption{Automizer structure for essential subgroups and $S$ in each corefree fusion system on a Sylow $3$-subgroup $S$ of $\O_8^+(3)$.}
    \end{table}
    
    \section{Fusion Systems on a Sylow $3$-subgroup of $\Fi_{23}$ and $\Bm$} \label{sec:fi23}

    From the ATLAS \cite{atlas}, we see that the following is an excerpt of the subgroup lattice of $\Bm$, up to conjugacy in $\Bm$:
    \begin{figure}[H]
        \centering
        \begin{tikzpicture}
            \node (A) at (0,0) {$\Fi_{23}$};
            \node (B) at (0,1) {$\Bm$};
            \node (C) at (-3,-1) {$\O^+_8(3) : \Sym(3)$};
            \node (D) at (-3,0) {$\O^+_8(3) : \Sym(4)$};

            \draw (A) -- (B);
            \draw (A) -- (C);
            \draw (B) -- (D);
            \draw (C) -- (D);
        \end{tikzpicture}
    \end{figure}
    \noindent Moreover, a Sylow $3$-subgroup of $\O_8^+(3) : \Sym(3)$ is also a Sylow $3$-subgroup of $\Bm$. Using the ATLAS \cite{atlas} again, we find the following $3$-local subgroups are contained in one of these three almost simple groups:
    \begin{align*}
        M_1 &\cong 3^{1+8}_+ : (2^{1+6}_- \ldotp \U_4(2) : 2) \leq \Bm \\
        M_2 &\cong 3^{2+3+6} : (\Sym(4) \times \GL_2(3)) \leq \Bm \\
        M_3 &\cong 3^{3+1+3+3} : \GL_3(3) \leq \Fi_{23} \\
        M_{4,i} &\cong 3^{3+6} : (\L_3(3) \times \Dih(8)) \leq \O_8^+(3) : \Sym(4) \\
        M_{5,i} &\cong 3^6 : \L_4(3) \leq \O_8^+(3) 
    \end{align*}
    
    Let $G := \Bm$, and fix a Sylow $3$-subgroup $S$ of $G$. We can choose a $G$-conjugate of $M_j$ and $M_{j,i}$ so that $S \in \Syl_3(M_j)$ for $1 \leq j \leq 3$ and $S \cap M_{j,i} \in \Syl_3(M_{j,i})$ for $4 \leq j \leq 5$ and $1 \leq i \leq 3$. We next construct further $3$-local subgroups of $G$:
    \begin{align*}
        N_{1,i} &:= M_1 \cap M_{5,i} \cong (3^{1+8}_+ : 3) : (\GL_2(3) \circ \GL_2(3)) \\
        N_2 &:= M_1 \cap M_3 \cong 3^{1+8}_+ : 3^{1+2}_+ : (2 \times \GL_2(3)) \\
        N_3 &:= M_1 \cap M_2 \cong 3^{1+8}_+ : 3^3 : (2 \times \Sym(4)) \\
        N_4 &:= M_2 \cap M_3 \cong 3^{3+1+3+3} : 3^2 : (2 \times \GL_2(3)) \\
        N_{5,i} &:= M_1 \cap M_{4,i} \cong 3^{1+8}_+ : 3^2 : (2 \times \GL_2(3))
    \end{align*}
    
    \begin{notation}
        Using the $3$-local subgroups of $G$ given above, we can define the following subgroups of $S$:
        \begin{itemize}
            \item $\bfW := O_3(M_1)$;
            \item $\bfJ := O_3(M_2)$;
            \item $\bfU := O_3(M_3)$;
            \item $\bfT_i := O_3(M_{4,i})$ ($1 \leq i \leq 3$);
            \item $\bfV_i := O_3(M_{5,i})$ ($1 \leq i \leq 3$);
            \item $\bfA_i := O_3(N_{1,i})$ ($1 \leq i \leq 3$);
            \item $P := O_3(N_2)$;
            \item $Q := O_3(N_3)$;
            \item $R := O_3(N_4)$; and
            \item $X_i := O_3(N_{5,i})$ ($1 \leq i \leq 3$).
        \end{itemize}
        For ease of notation, we write $\bfV_{i+3} := \bfV_i$, $X_{i+3} := X_i$ and so on.
    \end{notation}
    
    We can define the $3$-subgroups of $G$ based on the structure of $S$ as well:
    \begin{itemize}
        \item $\bfW$ is the preimage of $J(S/Z(S))$ in $S$;
        \item $\bfU = C_S(Z_3(S))$;
        \item $\bfT_i = C_S(D_i)$, where $D_i := \bfV_i \cap \bfV_{i+1}$.
        \item $\bfV_i$ are the three elementary abelian subgroups in $S$ of maximal order;
        \item $\bfA_i = X_i \cap X_{i+1}$;
        \item $P = C_S(Z_3(S)/Z(S))$;
        \item $Q$  is the unique maximal subgroup of $S$ such that $\mho(P) = Z_5(S)$;
        \item $R = C_S(Z_2(S))$; 
        \item $\bfJ = J(S) = Q \cap R$; and
        \item $X_i = C_Q(D_i/Z(S))$.
    \end{itemize}
    We see that most of these subgroups are characteristic in $S$, except for $\bfA_i$ and $X_i$. In fact, we have $N_S(X_i) = Q$ for $1 \leq i \leq 3$. We set $X := X_1$.

    Using the $3$-local subgroups of $G$ described above, we can recognize the almost simple subgroups of $G$ containing $S$:
    \begin{align*}
        G_0 &:= \langle N_G(P), N_G(R), N_G(X) \rangle \cong \Fi_{23} \\
        H_0 &:= \langle N_G(\bfV_i) \mid 1 \leq i \leq 3 \rangle \cong \O^+_8(3) : \Sym(3) \\
        H &:= \langle H_0, N_G(Q) \rangle \cong \O^+_8(3) : \Sym(4).
    \end{align*}
    We list the normalizers of $\bfW$ in these three subgroups:
    \begin{align*}
        N_{H_0}(\bfW) &\cong 3^{1+8}_+ : ((2^{1+6}_- : 3 \wr 3) : 2) \\
        N_H(\bfW) &\cong 3^{1+8}_+ : ((2^{1+6}_- : 3^3 : \Alt(4)): 2) \\
        N_{G_0}(\bfW) &\cong 3^{1+8}_+ : ((2^{1+6}_- : 3^{1+2}_+ : \SL_2(3)) : 2).
    \end{align*}
    For the other $3$-subgroups, the normalizers can be read off based on the normalizers of $P$, $Q$, $R$ and $X_i$ in the relevant almost simple group.
    
    We next consider the relationship between the normalizers in this case with that of $\O^+_8(3)$. We have
    \[K := O^{3}(O^{3'}(H_0)) = \langle O^{3'}(N_G(\bfV_i)) \mid 1 \leq i \leq 3 \rangle \cong \O^+_8(3).\]
    We see that the subgroup $Q$ is a Sylow $3$-subgroup of $K$. In Section \ref{sec:o83}, we saw that $\calF_Q(K)$ has four essential subgroups, namely $X_i$ ($1 \leq i \leq 3$) and $R_0 := R \cap Q$. The subgroups $\bfV_i$ and $\bfW$ are contained in $K$, and are centric radical in $\calF_Q(K)$.
    
    As $H$ and $H_0$ normalize $K$, $Q$ is strongly closed in $S$ with respect to $H$ and $H_0$. Note that
    \begin{align*}
        G &= \langle H, N_G(P) \rangle \\
        G_0 &= \langle H_0, N_G(P) \rangle.
    \end{align*}
    As $Q$ is not strongly closed in $N_G(P)$, we infer that $Q$ is neither strongly closed in $G$ nor in $G_0$. We further note that $N_K(P) = N_K(Q)$.

    \begin{proposition}
        Let $\calF := \calF_S(G) \cong \calF_S(\Bm)$. Then $\calE(\calF) = \{P,Q,R,X^S\}$ and $\calF^{cr} = \calE(\calF) \cup \{\bfA_1^S, \bfW, \bfJ, \bfT_1^S, \bfU, \bfV_1^S, S\}$.
    \end{proposition}
    \begin{proof}
        This follows from \cite[Theorem 2]{bm-radicals}.
    \end{proof}
    
    \begin{proposition}
        Let $\calF := \calF_S(G_0) \cong \calF_S(\Fi_{23})$. Then $\calE(\calF) = \{P,R,X^S\}$ and $\calF^{cr} = \calE(\calF) \cup \{\bfA_1^S, \bfW, \bfT_1^S, \bfU, \bfV_1^S, S\}$.
    \end{proposition}
    \begin{proof}
        This follows from \cite[Lemma 4.2]{fi23-radicals}.
    \end{proof}

    We define a further subgroup $Y$ of $S$. In Section \ref{sec:o83}, we saw that
    \[O^{3'}(\Out_K(\bfW)) \cong 2^{1+6}_- : 3^3 \cong \SL_2(3) \circ \SL_2(3) \circ \SL_2(3).\]
    Take a $3$-element $\gamma \in \Out_S(\bfW) \setminus \Out_Q(\bfW)$. Then $\gamma$ permutes the $3$ copies of $\SL_2(3)$ in $O^{3'}(\Out_K(\bfW))$. Set $Y$ to be the full preimage of $\langle \gamma \rangle$ in $S$. By construction, we have $[Y : \bfW] = 3$, $Y \cap Q = \bfW$ and $S = QY$. As $Y/\bfW$ has order $3$ and is not contained in $Q/\bfW \cong 3^3$, $Y/\bfW$ must be a subgroup of $P/\bfW \cong 3^{1+2}_+$. In particular, $[N_S(Y) : Y] = 3$.

    The subgroup $Y$ is essential in $\calF_S(H)$ and $\calF_S(H_0)$, with
    \[N_{H}(Y) = N_{H_0}(Y) \cong 3^{1+8}_+ : 3 : (2 \times \GL_2(3)).\]
    On the other hand, $Y$ is neither essential in $\calF_S(G)$ nor in $\calF_S(G_0)$, with
    \[N_{G}(Y) = N_{G_0}(Y) \cong 3^{1+8}_+ : 3 : (\Sym(3) \times \GL_2(3)).\]
    We see that $Y$ is no longer fully normalized in $\calF_S(G_0)$. This is because $Y$ is $G_0$-conjugate to some subgroup $Y_0 \leq P \cap Q$ containing $\bfW$. As $Q/\bfW$ is elementary abelian of order $3^3$, we infer that $N_S(Y_0) = Q$. As such,
    \[[N_S(Y_0) : Y_0] = 9 > 3 = [N_S(Y) : Y].\]
    Since the only subgroup in $\calE(\calF_S(G_0))$ that contains $Y$ is $P$, we see that maps in $\Aut_{G_0}(Y)$ must lift to $\Aut_G(P)$ by Alperin-Goldschmidt. We finally note that
    \[N_K(Y) \cong 3^{1+8}_+ : (2 \times \GL_2(3)).\]

    We appeal to GAP for the following results (see the file \texttt{fi23/o83-ess.g}).
    \begin{proposition}
        Let $\calF := \calF_S(H) \cong \calF_S(\O_8^+(3) : \Sym(4))$. Then $\calE(\calF) = \{R,X^S,Y^S\}$ and $\calF^{cr} = \calE(\calF) \cup \{\bfA_1^S, \bfW, \bfT_1^S, \bfV_1^S, S\}$.
    \end{proposition}
    
    \begin{proposition}
        Let $\calF := \calF_S(H_0) \cong \calF_S(\O_8^+(3) : \Sym(3))$. Then $\calE(\calF) = \{Q,R,X^S,Y^S\}$ and $\calF^{cr} = \calE(\calF) \cup \{\bfA_1^S, \bfW, \bfJ, \bfT_1^S, \bfV_1^S, S\}$.
    \end{proposition}

    We now let $\calF$ be a saturated fusion system on $S$. We construct $S$ in GAP by finding a Sylow $3$-subgroup of $\Fi_{23}$. $\Out_\calF(\bfW)$ can be found in MAGMA as a subgroup of $\Out(\bfW) \cong \Sp_8(3) : 2$.
    
    \begin{lemma} \label{lm:fi23_ess}
        We have $\calE(\calF) \subseteq \{P,Q,R,X^S,Y^S\}$.
    \end{lemma}
    \begin{proof}
        See Appendix \ref{sec:alg}.
    \end{proof}

    The code for Lemmas \ref{lm:fi23-w} to \ref{lm:fi23-u} can be found in the file \texttt{fi23/weak-closure.g}.

    \begin{lemma} \label{lm:fi23-w}
        The group $\bfW$ is weakly closed in $\calF$. In particular, $\calE(N_\calF(\bfW)) = \calE(\calF) \cap \{P, Q, X^S, Y^S\}$.
    \end{lemma}
    \begin{proof}
        The first part follows since $\bfW$ is the unique extraspecial subgroup of exponent $3$ in $S$ of order $3^{1+8}$. The second part follows from Lemma \ref{lm:essentials-in-normalizer}.
    \end{proof}
    
    \begin{lemma} \label{lm:fi23-j}
        The group $\bfJ$ is weakly closed in $\calF$. In particular, $\calE(N_\calF(\bfJ)) = \calE(\calF) \cap \{Q,R\}$.
    \end{lemma}
    \begin{proof}
        This follows since $J(S) = \bfJ$ by applying Lemmas \ref{lm:thom-weak-closed} and \ref{lm:essentials-in-normalizer}.
    \end{proof}
    
    \begin{lemma} \label{lm:fi23-u}
        The group $\bfU$ is weakly closed in $\calF$. In particular, $\calE(N_\calF(\bfU)) = \calE(\calF) \cap \{P,R\}$.
    \end{lemma}
    \begin{proof}
        We compute that 
        \[\bfU = C_S(Z_3(S)) = C_P(Z_2(P)) = C_R(Z_2(R)).\]
        Moreover, the only subgroups in $\calE(\calF)$ that contain $\textbf{U}$ are $P$ and $R$. We conclude that $\textbf{U}$ must be weakly closed in $\calF$. The second part follows from Lemma \ref{lm:essentials-in-normalizer}.
    \end{proof}
    
    The code for Lemmas \ref{lm:fi23-outfp} to \ref{lm:fi23-outfr} can be found in the file \texttt{fi23/automizer.g}.
    
    \begin{lemma} \label{lm:fi23-outfp}
        If $P \in \calE(\calF)$, then $P/\bfW\Phi(P)$ is a natural module for $O^{3'}(\Out_\calF(P)) \cong \SL_2(3)$.
    \end{lemma}
    \begin{proof}
        Assume that $r \in O^{3'}(\Aut_\calF(P))$ is a $3'$-element that centralizes $P/\bfW \Phi(P)$. Then $r$ also centralizes $P/\bfW$ as $[\bfW \Phi(P) : \bfW] = 3$. It follows that $[r, P, \bfW] \leq [\bfW, \bfW] = Z(P)$. Moreover, since $[[P, \bfW] : Z_4(P)] = 3$, we deduce that $[P, \bfW, r] \leq Z_4(P)$. By the three subgroups lemma, it follows that $[\bfW, r, P] \leq Z_4(P)$. As such, $[\bfW, r] \leq Z_5(P) = \Phi(P)$. By coprime action, we infer that $r = 1$. In that case, $O^{3'}(\Out_\calF(P))$ acts faithfully on $P/\bfW \Phi(P)$ of order $3^2$. We conclude that $P/\bfW \Phi(P)$ is a natural module for $O^{3'}(\Out_\calF(P)) \cong \SL_2(3)$.
    \end{proof}

    \begin{lemma} \label{lm:fi23-p-no-y}
        If $P \in \calE(\calF)$, then $Y$ is not fully $\calF$-normalized. In particular, $Y \not\in \calE(\calF)$.
    \end{lemma}
    \begin{proof}
        We recall that $[N_S(Y) : Y] = 3$. By Lemma \ref{lm:fi23-outfp}, we know that $O^{3'}(\Out_\calF(P)) \cong \SL_2(3)$ acts irreducibly on $P/\bfW \Phi(P)$. Let $\hat{Y} := \langle Y, \Phi(P) \rangle$. Then $\hat{Y}$ and $P \cap Q$ are two intermediate subgroups between $P$ and $\Phi(P) \bfW$. Thus, there exists an $\alpha \in O^{3'}(\Aut_\calF(P))$ such that $\hat{Y}\alpha = Q \cap P$. Set $Y_0 = Y\alpha$, which is a subgroup of $Q \cap P$. Then $N_S(Y_0) = Q$, meaning that $Y$ is not fully $\calF$-normalized.
    \end{proof}

    \begin{lemma} \label{lm:fi23-outfq}
        If $Q \in \calE(\calF)$, then $Q/\bfW$ and $Z_2(\bfJ)/Z(\bfJ)$ are natural modules for $O^{3'}(\Out_\calF(Q)) \cong \Alt(4)$.
    \end{lemma}
    \begin{proof}
        We compute that $[Q : \Phi(Q)] = 3^4$. As $Q/\bfW$ is elementary abelian of order $3^3$, we infer that $[\bfW : \Phi(Q)] = 3$. By coprime action, we deduce that $O^{3'}(\Out_\calF(Q))$ acts faithfully on $Q/\bfW$. Since $\Out_S(Q)$ has order $3$, $O^{3'}(\Out_\calF(Q))$ must be isomorphic to one of the following subgroups of $\SL_3(3)$: $\SL_2(3)$, $\Alt(4)$ or $13 : 3$. As $\Out_S(Q) \leq O^{3'}(\Out_{G}(Q)) \cong \Alt(4)$, $S$ acts indecomposably on $Q/\bfW$. This means that $Q/\bfW$ is a faithful, indecomposable $\GF(3)$-module of dimension $3$. Since $\Aut(Q)$ has no $13$-element, it follows that $O^{3'}(\Out_\calF(Q)) \cong \Alt(4)$.

        We next compute that $Z_2(\bfJ)$ is a self-centralizing subgroup of $S$ of order $3^5$. As such, Lemma \ref{lm:centric-faithful} tells us that $O^{3'}(\Out_\calF(Q))$ acts faithfully on $Z_2(\bfJ)$. As $Z(\bfJ) = Z_2(S)$ has order $3^2$, we find that $O_2(O^{3'}(\Out_\calF(Q)))$ must centralize $Z(\bfJ)$. By coprime action, we conclude that $O^{3'}(\Out_\calF(Q))$ acts naturally on $Z_2(\bfJ)/Z(\bfJ)$.
    \end{proof}

    \begin{lemma} \label{lm:fi23-outfr}
        If $R \in \calE(\calF)$, then both $R/\bfU$ and $Z(R)$ are natural modules for $O^{3'}(\Out_\calF(R)) \cong \SL_2(3)$.
    \end{lemma}
    \begin{proof}
        We compute that $[S, \bfU] = \Phi(R)$. This implies that $O^{3'}(\Out_\calF(R))$ centralizes $\bfU/\Phi(R)$. By coprime action, it follows that $O^{3'}(\Out_\calF(R))$ acts faithfully on $R/\bfU$ of order $3^2$. Thus, $R/\bfU$ is a natural module for $O^{3'}(\Out_\calF(R)) \cong \SL_2(3)$.

        We next compute that $Z_4(R)$ is elementary abelian of order $3^5$. Moreover, we have $[Z_{i+1}(R) : Z_i(R)] = 3$ for $1 \leq i \leq 3$. As $C_R(Z_4(R)) = Z_4(R)$, we find that $O^{3'}(\Out_\calF(R))$ acts faithfully on $Z_4(R)$ by Lemma \ref{lm:centric-faithful}. As such, if $r \in O^{3'}(\Aut_\calF(R))$ is a $3'$-element that centralizes $Z(R)$, then $r$ centralizes $Z_4(R)$, so that $r = 1$. This implies that $O^{3'}(\Out_\calF(R))$ acts faithfully on $Z(R)$ of order $3^2$. In particular, $Z(R)$ is a natural module for $O^{3'}(\Out_\calF(R))$.
    \end{proof}

    \begin{lemma} \label{lm:fi23-outfx}
        If $X \in \calE(\calF)$, then $O^{3'}(\Out_\calF(X)) \cong \SL_2(3)$.
    \end{lemma}
    \begin{proof}
        We recall that $N_S(X) = Q$, where $Q$ is isomorphic to a Sylow $3$-subgroup of $\O^+_8(3)$. So this follows from Lemma \ref{lm:o83-outfx}.
    \end{proof}
    
    \begin{lemma} \label{lm:fi23-pq-nfw}
        If $Q \in \calE(\calF)$ and $|\calE(N_\calF(\bfW))| \geq 2$, then $O_3(N_\calF(\bfW)) = \bfW$.
    \end{lemma}
    \begin{proof}
        Fix some $E \in \calE(N_\calF(\bfW)) \setminus \{Q\}$. Then we know that $\bfW \leq O_3(N_\calF(\bfW)) \leq E \cap Q < Q$ by Lemmas \ref{lm:weak-closure-to-normality} and \ref{lm:fi23-w}. But Lemma \ref{lm:fi23-outfq} tells us that $Q$ acts irreducibly on $Q/\bfW$, forcing $\bfW = O_3(N_\calF(\bfW))$.
    \end{proof}

    \begin{lemma} \label{lm:fi23-xy-nfw}
        If $X \in \calE(\calF)$ or $Y \in \calE(\calF)$, then $O_3(N_\calF(\bfW)) = \bfW$.
    \end{lemma}
    \begin{proof}
        First, assume that $X \in \calE(\calF)$. We know that $[X : \bfW] = 3^2$ and by Lemma \ref{lm:fi23-w} that $X \in \calE(N_\calF(\bfW))$ as well. Since $X$ is not normalized by $S$, we find that either $O_3(N_\calF(\bfW)) = \bfW$ or it is the preimage of a maximal subgroup of $X/\bfW$ in $S$. In the latter case, this subgroup must correspond to a normal subgroup of order $3$ in $S/\bfW \cong 3 \wr 3$. As such, the subgroup must equal the center of $S/\bfW$. But $X/\bfW$ does not contain $Z(S/\bfW)$, a contradiction.

        In the case that $Y \in \calE(\calF)$, this follows since $Y$ is not normal in $S$, $Y \in \calE(N_\calF(\bfW))$ and $[Y : \bfW] = 3$.
    \end{proof}
    
    \begin{proposition} \label{prp:fi23-OutFW}
        If $O_3(N_\calF(\bfW)) = \bfW$, then precisely one of the following holds:
        \begin{enumerate}
            \item $\calE(N_\calF(\bfW)) = \{X^S,Y^S\}$ and $O^{3'}(\Out_\calF(\bfW)) \cong 2^{1+6}_- : 3 \wr 3$;
            \item $\calE(N_\calF(\bfW)) = \{Q,X^S,Y^S\}$ and $O^{3'}(\Out_\calF(\bfW)) \cong 2^{1+6}_- : (3^3 : \Alt(4))$;
            \item $\calE(N_\calF(\bfW)) = \{P,X^S\}$ and $O^{3'}(\Out_\calF(\bfW)) \cong 2^{1+6}_- : (3^{1+2}_+ : \SL_2(3))$;
            \item $\calE(N_\calF(\bfW)) = \{P,Q,X^S\}$ and $O^{3'}(\Out_\calF(\bfW)) \cong 2^{1+6}_- \ldotp \U_4(2)$.
        \end{enumerate}
        Moreover, in each case, $O^{3'}(\Out_\calF(\bfW))$ acts irreducibly on $\bfW/Z(S)$.
    \end{proposition}
    \begin{proof}
        All code in this proof can be found in the file \texttt{fi23/outfw.m}. We note that $\bfW \cong 3^{1+8}_+$ and $G := O^{3'}(\Out_\calF(\bfW))$ satisfies the hypotheses given in Lemma \ref{lm:aut-sp8}. As such, the lemma tells us that $G$ must be isomorphic to one of the four given choices for $G$. We can then compute to find that the corresponding set $\calE(N_\calF(\bfW))$ for each of the four choices of automizers, noting that $\calE(N_\calF(\bfW)) \subseteq \{P, Q, X^S, Y^S\}$.
    \end{proof}
    
    \begin{lemma} \label{lm:fi23-x-then-y}
        If $X \in \calE(\calF)$ and $P \not\in \calE(\calF)$, then $Y \in \calE(\calF)$.
    \end{lemma}
    \begin{proof}
        Since $X \in \calE(\calF)$, we know by Lemma \ref{lm:fi23-xy-nfw} that $O_3(N_\calF(\bfW)) = \bfW$. But we have $P \not\in \calE(\calF)$, so Proposition \ref{prp:fi23-OutFW} implies that $Y \in \calE(\calF)$.
    \end{proof}
    \noindent We recall that Lemma \ref{lm:fi23-p-no-y} stated that if $P \in \calE(\calF)$, then $Y \not\in \calE(\calF)$. The lemma above serves as a partial converse to this statement.

    \begin{proposition} \label{prp:fi23-trivcore}
        Let $\calF$ be a saturated fusion system on $S$. Then either $\calF$ is constrained or we have $O_3(\calF) = 1$. Moreover, $O_3(\calF) = 1$ if and only if $\calE(\calF)$ contains $R$ and $O_3(N_\calF(\bfW)) = \bfW$.
    \end{proposition}
    \begin{proof}
        If $R \not\in \calE(\calF)$, then we know that $\bfW \normalIn \calF$ by Lemma \ref{lm:fi23-w}. Since $\bfW$ is self-centralizing in $S$, we see that $\calF$ is constrained. Now assume that $R \in \calE(\calF)$, but $\calE(N_\calF(\bfW)) = \{Q\}$ or $\calE(N_\calF(\bfW)) = \{P\}$. In the former case, we see that $\bfJ \normalIn \calF$ by Lemma \ref{lm:fi23-j}. In the latter case, we find that $\bfU \normalIn \calF$ by Lemma \ref{lm:fi23-u}. As $\bfJ$ and $\bfU$ are $S$-centric, we deduce again that $\calF$ is constrained. Combining with Lemmas \ref{lm:fi23-pq-nfw} and \ref{lm:fi23-xy-nfw}, we deduce that if $\calF$ is not constrained, then $\calE(\calF)$ contains $R$ and $O_3(N_\calF(\bfW)) =  \bfW$. The same also holds if $O_3(\calF) = 1$.

        Conversely, suppose that $R \in \calE(\calF)$ and $\bfW = O_3(N_\calF(\bfW))$. We know by Proposition \ref{prp:fi23-OutFW} that $O^{3'}(\Aut_\calF(\bfW))$ acts irreducibly on $\bfW/Z(S)$. This implies that $O_3(\calF)$ is an $\Aut_\calF(R)$-invariant subgroup of $Z(S)$. But we also know that $O^{3'}(\Aut_\calF(R))$ acts irreducibly on $Z_2(S)$ by Lemma \ref{lm:fi23-outfr}. This forces $O_3(\calF) = 1$, and we are done.
    \end{proof}

    \begin{lemma} \label{lm:fi23-autfs}
        Assume that $O_3(\calF) = 1$. Then $|\Aut_\calF(S)|_{3'} \geq 2^2$. Moreover, if $|\Aut_\calF(S)|_{3'} = 2^2$, then $P \not\in \calE(\calF)$.
    \end{lemma}
    \begin{proof}
        Assume first that $P \in \calE(\calF)$. As $O_3(\calF) = 1$, Proposition \ref{prp:fi23-OutFW} implies that $|\Aut_{O^{3'}(N_\calF(\bfW))}(S)|_{3'} \geq 2^2$. Since $R \in \calE(\calF)$ as well with $O^{3'}(\Out_\calF(R)) \cong \SL_2(3)$, there exists an involution $\alpha \in \Aut_\calF(S)$ such that $\alpha|_R \in O^{3'}(\Aut_\calF(R))$. By Lemma \ref{lm:fi23-outfr}, $\alpha$ cannot centralize $Z(S)$. As $Z(\bfW) = Z(S)$, it follows that $\alpha \not\in O^{3'}(N_\calF(\bfW))$. We deduce that $|\Aut_\calF(S)|_{3'} \geq 2^3$.

        Assume now that $P \not\in \calE(\calF)$. We know by Proposition \ref{lm:fi23-x-then-y} that $Y \in \calE(\calF)$. Using the structure of $O^{3'}(\Out_\calF(\bfW))$ in Proposition \ref{prp:fi23-OutFW}, we deduce that $O^{3'}(\Out_\calF(Y)) \cong \SL_2(3)$. By the Extension Axiom, we know that there exists an involution $\beta \in \Aut_\calF(S)$ such that $\beta|_{Y} \in N_{O^{3'}(\Aut_\calF(Y))}(\Aut_S(Y))$. Since $\beta$ centralizes $Z(Y) = Z(S)$, $\alpha$ and $\beta$ cannot be $\Aut(S)$-conjugate. This shows that $|\Aut_\calF(S)|_{3'} \geq 2^2$.
    \end{proof}

    Using all the information we have gathered in this section, we will classify all saturated fusion systems $\calF$ on $S$ such that $O_3(\calF) = 1$.

    \begin{theorem} \label{thm:fi23-o83}
        Assume that $O_3(\calF) = 1$ and $|\Aut_\calF(S)|_{3'} = 2^2$. Then $\calF$ is realized by $\O^+_8(3) : 3$ or $\O^+_8(3) : \Alt(4)$.
    \end{theorem}
    \begin{proof}
        All code used in this proof can be found in the file \texttt{fi23/foc.g} and \texttt{fi23/foc.m}. We make use of Lemma \ref{lm:hyp-f} to show that $Q = \foc(\calF)$. We recall by Lemma \ref{lm:fi23-autfs} that $P \not\in \calE(\calF)$.
        
        Let $E \in \calE(\calF)$. If $E \in \{ Q, X \}$, then we have $[E, O^{3'}(\Aut_\calF(E))] \leq E \leq Q$. If we instead have $E = Y$, then as $[Y : \bfW] = 3$, it follows that $[Y, O^{3'}(\Aut_\calF(Y))] \leq \bfW \leq Q$. 

        We now assume that $E = R$. By coprime action, we know that
        \[R/\Phi(R) = [R/\Phi(R), O^{3'}(\Aut_\calF(R))] \times C_{R/\Phi(R)}(O^{3'}(\Aut_\calF(R))).\]
        By Lemma \ref{lm:fi23-outfr}, we have
        \[C_{R/\Phi(R)}(O^{3'}(\Aut_\calF(R))) = \bfU/\Phi(R).\]
        As $Q \cap R \cap \bfU = \Phi(R)$ and $\langle Q \cap R, \bfU \rangle = R$,
        \[[R/\Phi(R), O^{3'}(\Aut_\calF(R))] = (Q \cap R)/\Phi(R).\]
        We conclude that $[R, O^{3'}(\Aut_\calF(R))] \leq Q$. 

        We next consider $\Aut_\calF(S)$. Let $\alpha, \beta \in \Aut_\calF(S)$ be as given in Lemma \ref{lm:fi23-autfs}. Since $\alpha|_R \in O^{3'}(\Aut_\calF(R))$ and $\beta|_Y \in O^{3'}(\Aut_\calF(Y))$, we deduce that $[R, \alpha] \leq Q$ and $[Y, \beta] \leq Q$. Since $Q$ is not contained inside $R$ and $Y$, we conclude that $[S, \alpha] \leq Q$ and $[S, \beta] \leq Q$. As $|\Aut_\calF(S)|_{3'} = 2^2$, we have that $\Aut_\calF(S) = \langle \Inn(S), \alpha, \beta \rangle$. This implies that $[S, \Aut_\calF(S)] \leq Q$. 
        Using Lemma \ref{lm:fi23-outfx}, we deduce that $[X, O^{3'}(\Aut_\calF(X))] = \bfW$. As $[R, O^{3'}(\Aut_\calF(R))] = Q \cap R$, it follows that
        \[\foc(\calF) = \langle [P, \phi] \mid P \leq S, \phi \in \Aut_\calF(P) \rangle = Q.\]
        In particular, $\hyp(\calF) \leq Q$. 
        
        By Proposition \ref{prp:fi23-OutFW}, we know the structure of $O^{3'}(\Out_\calF(\bfW))$. We compute that $O_2(\Out_\calF(\bfW)) \cong 2^{1+6}_-$ acts irreducibly on $\bfW/Z(S)$. As such, $\bfW \leq \hyp(\calF)$. Since $S/\bfW \cong 3 \wr 3$, we have that $\bfW \Phi(S)/\bfW = Z_2(S/\bfW)$. As $\Phi(S) \hyp(\calF) = Q$, it follows that $\hyp(\calF) = Q$.
        
        Set $\calG := O^3(\calF)$. We know that the fusion system $\calG$ is defined on $\hyp(\calF) = Q$. We recall that $Q$ is a Sylow $3$-subgroup of $\O^+_8(3)$. Moreover, since $O_3(\calF) = 1$, we know by \cite[Proposition 5.47]{cra} that $O_3(\calG) = 1$. By Theorem \ref{thm:o83}, we infer that $O^{3'}(\calG)$ is realized by $\O^+_8(3)$. As such, \cite[Theorem A]{or-fsg} tells us that $O^{3'}(\calG)$ is simple. Then we infer by \cite[Theorem C]{bmor} that $\calF$ is realized by some group $H$ with $F^*(H) \cong \O^+_8(3)$. As $|\Aut_\calF(S)|_{3'} = 2^2$, we conclude that $\calF$ is realized by $\O^+_8(3) : 3$ or $\O^+_8(3) : \Alt(4)$.
    \end{proof}

    \begin{theorem} \label{thm:fi23-b}
        Let $\calF$ be a saturated fusion system on $S$ such that $|\Aut_\calF(S)|_{3'} = 2^3$ and $O_3(\calF) = 1$.
        \begin{enumerate}
            \item If $Q \not\in \calE(\calF)$, then $\calF$ is realized by $\O^+_8(3) : \Sym(3)$ or $\Fi_{23}$.
            \item If $Q \in \calE(\calF)$, then $\calF$ is realized by $\O^+_8(3) : \Sym(4)$ or $\Bm$.
        \end{enumerate}
    \end{theorem}
    \begin{proof}
        All code in this proof can be found in the file \texttt{fi23/uniqueness-r.g}, \texttt{fi23/uniqueness-q.g} and \texttt{fi23/uniqueness-w.m}. 

        We define the group $G_0$ containing $S$ as follows:
        \[G_0 := \begin{cases}
            \O^+_8(3) : \Sym(3), & \calE(\calF) = \{R, X^S, Y^S\} \\
            \O^+_8(3) : \Sym(4), & \calE(\calF) = \{Q, R, X^S, Y^S\} \\
            \Fi_{23}, & \calE(\calF) = \{P, R, X^S\} \\
            \Bm, & \calE(\calF) = \{P, Q, R, X^S\}.
        \end{cases}\]
        Combining Propositions \ref{prp:fi23-OutFW} and \ref{prp:fi23-trivcore}, we know that these are precisely the four choices for $\calE(\calF)$. Set $\calG := \calF_S(G_0)$, so that $\calE(\calF) = \calE(\calG)$. Moreover, we find that $|\Aut_\calF(S)|_{3'} = |\Aut_\calG(S)|_{3'} = |\Aut(S)|_{3'} = 2^3$. By Sylow's Theorems, we may conjugate $\calG$ by some $\alpha \in \Aut(S)$ so that $\Aut_\calF(S) = \Aut_\calG(S)$.
        
        We now consider the choices for $\Aut_\calG(R)$. We compute using GAP that, for a fixed choice of $\Aut_{\Aut_\calG(S)}(R)$, there are precisely three choices for $\Out_\calG(R)$ in $\Out(R)$. Moreover, these three choices are $\Aut(S)$-conjugate. As such, we may further assume $\Aut_\calF(R) = \Aut_\calG(R)$.

        If $Q \not\in \calE(\calF)$, then 
        \[N_\calF(Q) = \langle \Aut_\calF(S) \rangle_S = \langle \Aut_\calG(S) \rangle_S = N_\calG(Q).\]
        On the other hand, if $Q \in \calE(\calF)$, then we compute using GAP that
        \[N_{\Out(Q)}(N_{\Out_\calG(Q)}(\Out_S(Q))) \leq N_{\Out(Q)}(\Out_\calG(Q)).\]
        Furthermore, there is a unique conjugacy class in $\Out(Q)$ of $\Out_\calF(Q) \cong 2 \times 2 \times \Sym(4)$. This implies that $\Aut_\calF(Q) = \Aut_\calG(Q)$. We deduce that
        \[N_\calF(Q) = \langle \Aut_\calF(S), \Aut_\calF(Q) \rangle_S = \langle \Aut_\calG(S), \Aut_\calG(Q) \rangle_S = N_\calG(Q).\]
        
        We compute using MAGMA that
        \[N_{\Out(\bfW)}(N_{\Out_\calG(\bfW)}(\Out_S(\bfW))) \leq N_{\Out_\calG(\bfW)}(\Out_S(\bfW)) \leq N_{\Out(\bfW)}(\Out_\calG(\bfW)).\]
        We further appeal to MAGMA to infer that there is a unique conjugacy class in $\Out(\bfW) \cong \Sp_8(3) : 2$ of subgroups isomorphic to $O^{3'}(\Out_\calF(\bfW))$. Using Lemma \ref{lm:weak-closure-equality}, it follows that $O^{3'}(\Aut_\calF(\bfW)) = O^{3'}(\Aut_\calG(\bfW))$. By the Frattini Argument and Lemma \ref{lm:autfe-decomposition}, we further see that $\Aut_\calF(\bfW) = \Aut_\calG(\bfW)$. We next compute that the map
        \[H^1(\Out_\calG(\bfW); Z(\bfW)) \to H^1(\Out_{N_\calG(Q)}(\bfW); Z(\bfW))\]
        is surjective. By \cite[Proposition 2.11]{todd-modules}, we conclude that $N_\calF(\bfW) = N_\calG(\bfW)$.

        Thus, we have shown that
        \[\calF = \langle \Aut_\calF(S), \Aut_\calF(R), N_\calF(\bfW) \rangle_S = \langle \Aut_\calG(S), \Aut_\calG(R), N_\calG(\bfW) \rangle_S = \calG,\]
        as desired.
    \end{proof}

    If $O_3(\calF) = 1$, then we saw in Lemma \ref{lm:fi23-autfs} that $|\Aut_\calF(S)|_{3'} \geq 2^2$. We dealt with the case when $|\Aut_\calF(S)|_{3'} = 2^2$ in Theorem \ref{thm:fi23-o83}, and the case when $|\Aut_\calF(S)|_{3'} = 2^3$ in Theorem \ref{thm:fi23-b}. As $|\Aut(S)|_{3'} = 2^3$, the following result holds.
    \begin{theorem} \label{thm:fi23}
        Let $\calF$ be a saturated fusion system on $S$ with $O_3(\calF) = 1$. Then $\calF$ is realized by an almost simple group $G$, where precisely one of the following holds:
        \begin{enumerate}
            \item $O^{3'}(G) = F^*(G)$, and $G$ is isomorphic to $\Fi_{23}$ or $\Bm$; or
            \item $O^{3'}(G)$ is isomorphic to $\O_8^+(3) : 3$ or $\O_8^+(3) : \Alt(4)$.
        \end{enumerate}
    \end{theorem}

    Table \ref{tbl4} summarizes the structure of the automizers in the six corefree fusion systems on $S$.
    \begin{table} \label{tbl4}
        \centering
        \resizebox{\textwidth}{!}{\begin{tabular}{c|c|c|c|c|c|c}
            $\Out_\calF(P)$ & $\Out_\calF(Q)$ & $\Out_\calF(R)$ & $\Out_\calF(X)$ & $\Out_\calF(Y)$ & $\Out_\calF(S)$ & $G$ \\
            \hline
            $-$ & $-$ & $\GL_2(3)$ & $\GL_2(3)$ & $\GL_2(3)$ & $2 \times 2$ & $\O_8^+(3) : 3$ \\
            $-$ & $2 \times 2 \times \Alt(4)$ & $\GL_2(3)$ & $2 \times 2 \times \GL_2(3)$ & $\GL_2(3)$ & $2 \times 2$ & $\O_8^+(3) : \Alt(4)$ \\
            $-$ & $-$ & $2 \times \GL_2(3)$ & $2 \times \GL_2(3)$ & $2 \times \GL_2(3)$ & $2 \times 2 \times 2$ & $\O_8^+(3) : \Sym(3)$ \\
            $-$ & $2 \times 2 \times \Sym(4)$ & $2 \times \GL_2(3)$ & $\Dih(8) \times \GL_2(3)$ & $2 \times \GL_2(3)$ & $2 \times 2 \times 2$ & $\O_8^+(3) : \Sym(4)$ \\
            $2 \times \GL_2(3)$ & $-$ & $2 \times \GL_2(3)$ & $2 \times \GL_2(3)$ & $-$ & $2 \times 2 \times 2$ & $\Fi_{23}$ \\
            $2 \times \GL_2(3)$ & $2 \times 2 \times \Sym(4)$ & $2 \times \GL_2(3)$ & $\Dih(8) \times \GL_2(3)$ & $-$ & $2 \times 2 \times 2$ & $\Bm$ 
        \end{tabular}}
        \caption{Automizer structure for essential subgroups and $S$ in each corefree fusion system on a Sylow $3$-subgroup $S$ of $\Fi_{23}$.}
    \end{table}

    \section*{Acknowledgments}
    This work took place during the author's PhD at the University of Manchester. The author acknowledges the support received from EPSRC (EP/W524347/1) during this period.
    
    The author thanks Dr. Martin van Beek for his supervision and constant support. The author also thanks Prof. Charles Eaton for proofreading this manuscript. The author would also like to thank Prof. Max Horn for helping to implement the optimization version of the algorithm for finding proto-essential subgroups in GAP that was used in this project. The author finally thanks the referee for reading the paper carefully.

    \appendix

    \section{Proto-essential subgroups} \label{sec:alg}
    In this section, we define what it means for a subgroup $E$ of a $p$-group $S$ to be proto-essential. The file \texttt{find-proto-essentials.g} in the GitHub package implements the tests we mention. We run this code on all subgroups of $S$ in this paper, up to $\Aut(S)$-conjugacy.

    We say that a subgroup $E$ of $S$ is \emph{proto-essential} if it passes the following tests:
    \begin{enumerate}
        \item \textbf{Centric Test}. Check that $C_S(E) \leq E$.
        \item \textbf{Rank Test}. Check that $\Out_S(E)$ is a Sylow $p$-subgroup of a group containing a strongly $p$-embedded subgroup. Moreover, if this is the case, then use \cite[Theorem 6.9]{sambale} to ensure that the rank of $E$ is sufficiently large.
        \item \textbf{Frattini Test}. Check that $[N_S(E),E] \nleq \Phi(E)$ and that $C_{N_S(E)}(E/\Phi(E)) \leq E$.
        \item \textbf{Radical Test}. Check that $\Out_S(E) \cap O_p(\Out(E)) = 1$.
        \item \textbf{Lifting Test}. If $\Out_S(E)$ is elementary abelian of order $p^n > p$ and $\Aut(E)$ does not have a section isomorphic to $\Alt(2p)$ (if $p \geq 5$), check that $\Aut(N_S(E))$ has an element of order $\frac{1}{(2,p^n-1)} (p^n-1)$.
    \end{enumerate}

    We briefly justify these steps. Steps (1) and (4) follow from the definition of a radical subgroup. If $E < C_{N_S(E)}(E/\Phi(E)) =: A$ and $C_S(E) \leq E$, then
    \[\Inn(E) < \Aut_A(E) \leq C_{\Aut_S(E)}(E/\Phi(E)) \leq O_p(\Aut(E))\]
    by Lemma \ref{lm:burnside}. This means that $E$ cannot be $S$-radical. Performing this check before step (4) ensures that we do not compute $\Aut(E)$, an expensive process. As an optimization, we first compute whether $[N_S(E), E] \leq \Phi(E)$ before computing the centralizer itself. If $[E : \Phi(E)] \geq p^6$, then step (4) might be too expensive to run. In these cases, $\Out_S(E)$ is typically elementary abelian of order $p^n > p$. As such, step (5) can help identify the subgroups of interest in these cases. The following result justifies step (5).
    \begin{lemma}
        Let $S$ be a finite $p$-group, and let $E \leq S$ be such that $N_S(E)/E$ is elementary abelian of order $q := p^n$, with $p$ odd and $n \geq 2$. Let $\calF$ be a fusion system on $S$ such that $E \in \calE(\calF)$. Then:
        \begin{enumerate}
            \item $n = 2$ and $p \geq 5$, and $O^{p'}(\Out_\calF(E))/O_{p'}(\Out_\calF(E))$ is isomorphic to $\Alt(2p)$ or $\Fi_{22}$ (for $p=5$), or
            \item $\Out_\calF(N_S(E))$ contains an element of order $\frac{1}{(2, q-1)}(q-1)$.
        \end{enumerate}
    \end{lemma}
    \begin{proof}
        We know that $\Out_S(E)$ is elementary abelian of order $p^n > p$. Using \cite[Proposition 4.5]{algorithms}, we have the following options for $K := O^{p'}(\Out_\calF(E)/O_{p'}(\Out_\calF(E)))$:
        \begin{enumerate}
            \item $K \cong \L_2(q)$;
            \item $q = p^2$, $p \geq 5$ and $K \cong \Alt(2p)$;
            \item $q = 3^2$ and $K \cong \L_3(4)$;
            \item $q = 3^2$ and $K \cong \Mt_{11}$;
            \item $q = 5^2$ and $K \cong {^2 \textrm{F}_4(2)}'$; and
            \item $q = 5^2$ and $K \cong \Fi_{22}$.
        \end{enumerate}
        In cases (2) and (6), $K$ contains a section isomorphic to $\Alt(2p)$. Otherwise, there exists a subgroup $K_0 \cong \L_2(q)$ of $K$ such that $\Syl_p(K_0) \subseteq \Syl_p(K)$. We further see that $N_{K_0}(T_0)/T_0$ is cyclic of order $a_0 := \frac{1}{(2,q-1)}(q-1)$. Let $t_0 \in N_{K_0}(T_0)/T_0$ generate this group, and take $t' \in K$ to be such that $t'T_0 = t_0$. Fix $t \in O^{p'}(\Aut_\calF(E))$ to satisfy $t O_{p'}(\Out_\calF(E)) = t'$. 
        
        By construction, $t$ has order a multiple of $a_0$ and normalizes some Sylow $p$-subgroup of $O^{p'}(\Out_\calF(E))$. Since $E \in \calE(\calF)$, we know that $\Out_S(E)$ is a Sylow $p$-subgroup of $\Out_\calF(E)$. As such, we may assume that $t$ normalizes $\Out_S(E)$. Applying Lemma \ref{lm:autfe-decomposition}, we conclude that $\Out_\calF(N_S(E))$ contains an element of order $a_0$.
    \end{proof}

    We now consider how the algorithm performs for the four $3$-groups considered in this paper. Table \ref{tbl5} lists the number of subgroups that pass each case.
    \begin{table} \label{tbl5}
        \centering
        \begin{tabular}{c|c|c|c|c}
            $G$ & $\Fi_{22}$ & $\TE : 3$ & $\O^+_8(3)$ & $\Fi_{23}$, $\Bm$ \\
            \hline
            $|S|$ & $3^9$ & $3^{10}$ & $3^{12}$ & $3^{13}$ \\
            \hline
            Total subgroups & $986$ & $2 \ 975$ & $49 \ 066$ & $177 \ 864$ \\
            Centric Test & $233$ & $806$ & $4 \ 061$ & $11 \ 894$ \\
            Rank Test & $128$ & $301$ & $980$ & $2 \ 620$ \\
            Frattini Test & $50$ & $87$ & $135$ & $324$ \\
            Radical Test & $7$ & $5$ & $2$ & $6$ \\
            Lift Test & $3$ & $4$ & $2$ & $6$ \\
            \hline
            Union of $\calE(\calF)$ & $3$ & $4$ & $2$ & $5$ 
        \end{tabular}
        \caption{The number of subgroups of a Sylow $p$-subgroup $S$ of $G$ that pass each step of the proto-essential test. The subgroups are computed up to $\Aut(S)$-conjugacy. The number of subgroups that are essential in some known fusion system $\calF$ of $S$ is also provided.}
    \end{table}

    We see that, except for $\Fi_{23}$ and $\Bm$, the proto-essential subgroups are precisely the subgroups that are essential in some fusion system on $S$ realized by an almost simple group. In the $\Fi_{23}$ and $\Bm$ case, let $E$ be the proto-essential subgroup that is not essential in $\Fi_{23}$, $\Bm$, $\O^+_8(3) : 3$ or $\O^+_8(3) : \Sym(4)$. Then we see that no intermediate subgroup $\Out_S(E) \leq A \leq \Out(E)$ is such that $O_3(A) = 1$, meaning that $E$ cannot be essential.

    \bibliographystyle{alpha}
    \bibliography{nat}

\end{document}